\documentclass[10pt,reqno]{amsart}

\usepackage{amsmath,amssymb,esint,mathrsfs}
\usepackage{xcolor}
\usepackage[colorlinks=true]{hyperref}

\allowdisplaybreaks
\numberwithin{equation}{section}

\newtheorem{lemma}{Lemma}[section]
\newtheorem{theorem}[lemma]{Theorem}
\newtheorem{proposition}[lemma]{Proposition}
\newtheorem{corollary}[lemma]{Corollary}

\newtheorem{remark}[lemma]{Remark}

\def\R{{\mathbb R}}
\def\Rn{{\mathbb R}^n}
\def\N{{\mathbb N}}

\def\eps{{\varepsilon}}
\def\ii{{\tt i}}
\def\kk{{\tt k}}
\def\Ieps{{\mathcal I}_j}
\def\xiij{x_j^{\ii}}
\def\rj{\eps_j}
\def\Bijh{B_j^{\ii,h}}
\def\Bijhj{B_j^{\ii,h_j}}
\def\Cijh{C_j^{\ii,h}}
\def\Cij{C_j^{\ii,h_j}}
\def\Vij{V_j^{\ii}}
\def\Ln{\mathcal{L}^n}
\newcommand{\defectj}{\mathcal D_j}
\newcommand{\Dhom}{\dot{W}}
\newcommand{\capsp}{\operatorname{cap}_{s,p}}
\newcommand{\capp}{\operatorname{cap}_{p}}
\newcommand{\Capsp}{C_{s,p}}
\newcommand{\Capp}{C_p}

\title[BBM localization and obstacle homogenization]
{The Bourgain, Brezis, Mironescu localization\\ for
deterministic obstacle problems}

\author[Matteo Focardi]{Matteo Focardi,\\
DiMaI ``Ulisse Dini'', V.le Morgagni 67/A\\
50134 - Firenze - Italy}
\date{\today}

\begin{document}

\maketitle

\begin{abstract}
We study the asymptotic behaviour of families of fractional energies normalized according to the Bourgain--Brezis--Mironescu formula and defined on functions satisfying a pinning condition on possibly aperiodic families of small obstacles. The result of the simultaneous limit depends 
crucially on the critical rescaling chosen for the obstacles.
\end{abstract}

\vspace{0.3cm}
\begin{flushright}
    {\itshape This work is dedicated to Paolo Marcellini, a pioneer in homogenization, on the occasion of his 80th birthday}
\end{flushright}
\vspace{0.3cm}

\section{Introduction}\label{s:introduction}

In the seminal paper \cite{BBM}, Bourgain, Brezis, and Mironescu established a connection between classical Sobolev spaces and their fractional counterparts as the order of fractional differentiability tends to one (cf. also \cite{Brezis}).  
In particular, the appropriate normalization of the fractional seminorm restores continuity in the differentiation index.
More precisely, let $U\subset\Rn$ be a bounded, connected Lipschitz open set, $n\geq2$, and $p\in(1,\infty)$.  For $u\in L^p(U)$ and $s\in(0,1)$, set
\[
 [u]_{s,p;U}^p:=\int_{U\times U}
       \frac{|u(x)-u(y)|^p}{|x-y|^{n+sp}}\,dx\,dy.
\]
Then \cite[Corollary~2]{BBM} gives
\[
 \lim_{s\to1^-}(1-s)[u]_{s,p;U}^p
 =\begin{cases}
 K_{n,p}\displaystyle\int_U|\nabla u|^p\,dx,
       &u\in W^{1,p}(U),\\
 +\infty,&u\in L^p(U)\setminus W^{1,p}(U).
 \end{cases}
\]
The constant $K_{n,p}$ is given explicitly in \eqref{e:Knp}.

Subsequent developments have extended this result in several directions: to the whole space $\Rn$, to variational convergence in the sense of $\Gamma$-convergence, to problems with boundary conditions, and to magnetic Sobolev spaces.  Necessary and sufficient conditions on more general families of kernels have also been investigated (see, among others, \cite{Brezis,PonceGamma,NguyenPinamontiSquassinaVecchi,
CrismaleDeLucaKubinNinnoPonsiglione,KubinSaraccoStefani,DavoliDiFrattaPagliari,
GennaioliStefani}).

Recent contributions have analyzed the interaction of the Bourgain, Brezis, and Mironescu limit (BBM limit in short) with another singular limit involving microscopic scales: homogenization.
For quadratic fractional energies with periodic coefficients,
Braides, Brusca, and Donati \cite{BraidesBruscaDonati} established a regime in which localization and homogenization can be separated (see also \cite{FocardiSurvey} for a preliminary result).  Braides, Mannella, and Schiavoni Piazza \cite{BraidesMannellaSchiavoniPiazza} identified the transition scale between homogenized and averaged energies, governed by the relative rates of the fractional order approaching one and the period tending to zero.

In this paper we pursue a related but different path by studying the asymptotics 
via $\Gamma$-convergence of the rescaled Gagliardo seminorms
under the additional constraint that the admissible functions vanish on a family of
small obstacles. 
  
The limiting behavior of solutions to local Dirichlet problems in periodically perforated domains has been an active area of research since the 1970s, following the seminal papers of Marchenko and Khruslov \cite{marchenko-khruslov}, Rauch and Taylor \cite{rauch-taylor-1,rauch-taylor-2},  Cioranescu and Murat \cite{cioranescu-murat-2,cioranescu-murat-1}, and De Giorgi, Dal Maso and Longo \cite{DeGiorgiDalMasoLongo} (cf. \cite{DalMaso-bilateral,DalMaso-unilateral,DalMaso-integralrepr} for a complete solution of the problem). 
These problems are motivated by applications in electrostatic screening, fluid mechanics, elasticity and acoustics, among others.
It is well known that when the perforations have a critical size the limit equation or functional contains, in addition to the original operator or energy, a zero-order term reflecting the cumulative capacitary cost of the obstacles. 
We do not attempt to survey the vast bibliography on the subject and refer to the comprehensive accounts in 
\cite{marchenko-khruslov,DalMaso-book, Attouch,BraidesDeFranceschi,BraidesBeg}.
Here, we only stress that the capacity of the individual obstacles, together with their distribution, determines their contribution to the limit problem.

The interest in the homogenization of obstacle problems resurfaced more recently thanks to the work of Caffarelli and Mellet \cite{caffarelli-mellet} in the nonlocal setting, with interest driven by applications in contact mechanics, flow through semipermeable membranes, and mathematical finance among others.
Caffarelli and Mellet studied these problems for the fractional Laplacian via Tartar's oscillating test functions method and taking advantage of the celebrated Caffarelli--Silvestre extension argument \cite{caffarelli-silvestre} to reduce the nonlocal homogenization problem to one involving a local but degenerate energy.
An intrinsic variational approach working directly with the nonlocal energy was later proposed by the author, using $\Gamma$-convergence techniques for an aperiodic distribution of obstacles \cite{FocardiAperiodic} (see also \cite{FocardiCPDE} for an earlier result using $\Gamma$-convergence and the above mentioned extension method, and \cite{Palatucci1,Palatucci2} for soft Robin-type obstacle constraints and a different intrinsic approach if $p=2$).

Stochastic distributions of obstacles satisfying suitable stationarity and integrability conditions, allowing dependence between centres and marks as well as clustering of the obstacles have been recently  dealt with in the work of Giunti, H\"ofer and Vel\'azquez \cite{GiuntiHoferVelazquez} in the linear local case using PDEs methods.
The nonlinear counterpart using the variational approach has been studied in Scardia, Zemas and Zeppieri \cite{ScardiaZemasZeppieri} still in the local setting. Deangelis, Zeppieri and the author \cite{DeangelisFocardiZeppieri} have established more recently an extension to fractional energies
under natural ergodicity assumptions.

\subsection{Overview of the main results}

The focus of this paper is to consider simultaneously the BBM limit and the homogenization 
of an obstacle condition, with deterministic though possibly aperiodic distributions.  
For these problems, as customary, we assume $p\in(1,n)$.
We stress that both asymptotic processes involve a localization mechanism, so their interaction requires separating two contributions which concentrate near the diagonal.

For $p=2$, the associated minimum problems describe diffusion in a medium containing small absorbing inclusions.  The simultaneous limit asks whether the collective absorption due to these inclusions persists as the diffusion energy becomes local, and how its strength depends on their size.

The outcome of our analysis is that the interaction between the two phenomena depends subtly on the chosen scaling of the obstacles. In particular, if
$s_j\to1^-$ denotes the fractional order of differentiability and $\eps_j\to0^+$ the characteristic distance between the obstacles, at the critical scale
$\eps_j^{\frac n{n-s_jp}}$, dictated by the homogeneity of the fractional capacity, the bulk and capacitary terms are additive and their coefficients are independent of the relative rates at which $1-s_j$ and $\eps_j$ vanish (Theorem~\ref{t:aperiodic-BBM}).  
Instead, at the critical scale $\eps_j^{\frac n{n-p}}$, dictated by the capacity of the limiting local energy, three different regimes arise (Theorem~\ref{t:aperiodic-local-critical-scale}).
For further details on the interaction between the two concentration mechanisms we refer to the discussion in Subsection~\ref{ss:heuristics}.

For the sake of simplicity, in this introduction we describe the results we obtain for a periodic distribution of obstacles, referring to Theorems~\ref{t:aperiodic-BBM} and~\ref{t:aperiodic-local-critical-scale} for the general, possibly aperiodic, statements.  Extensions to thin obstacles and to translation-invariant homogeneous kernels are discussed in Remarks~\ref{r:capacity-constraint} and~\ref{r:homogeneous-kernels}, respectively. Furthermore, the ideas behind the proof of the results described below will be discussed heuristically in Subsection~\ref{ss:heuristics}.

Let then $T\Subset B_1$ be a nonempty open set, and let
$s_j\uparrow1$ and $\eps_j\downarrow0$.  Set
\begin{equation}\label{e:obstacle-set-intro}
 \lambda_j:=\eps_j^{\frac{n}{n-s_jp}},\qquad
 \mathcal K_j:=\bigcup_{\ii\in\mathbb Z^n}(\eps_j\ii+\lambda_jT),
\end{equation}
and define $\mathcal F_j:L^p(U)\to[0,+\infty]$ by
\begin{equation}\label{e:aperiodic-functionals-intro}
 \mathcal F_j(u):=
 \begin{cases}
 (1-s_j)[u]_{s_j,p;U}^p,
 &u\in W^{s_j,p}(U),\quad u=0\text{ a.e. on }\mathcal K_j\cap U,\\
 +\infty,&\text{otherwise}.
 \end{cases}
\end{equation}
Under the above assumptions $(\mathcal F_j)_j$ $\Gamma$-converges in $L^p(U)$ to
\begin{equation}\label{e:aperiodic-limit-functional-intro}
 \mathcal F(u):=\begin{cases}
 K_{n,p}\displaystyle\int_U|\nabla u|^p\,dx
 +K_{n,p}\capp(T)\displaystyle\int_U|u|^p\,dx,
 &u\in W^{1,p}(U),\\
 +\infty,&\text{otherwise}.
 \end{cases}
\end{equation}
Here, $\capp(T)$ denotes the homogeneous $p$-capacity defined in
\eqref{e:local-capacity}.  

Theorem~\ref{t:aperiodic-local-critical-scale} deals instead with the
rescaling associated with the local capacity.  In the periodic case, set
\begin{equation}\label{e:local-critical-obstacle-scale-intro}
 \bar\lambda_j:=\eps_j^{\frac{n}{n-p}},
 \qquad
 \bar{\mathcal K}_j:=
 \bigcup_{\ii\in\mathbb Z^n}(\eps_j\ii+\bar\lambda_jT),
\end{equation}
and define $\bar{\mathcal F}_j:L^p(U)\to[0,+\infty]$ by
\begin{equation}\label{e:local-scale-functionals-intro}
 \bar{\mathcal F}_j(u):=
 \begin{cases}
  (1-s_j)[u]_{s_j,p;U}^p,
  &u\in W^{s_j,p}(U),\quad
   u=0\quad\text{a.e. on }
   \bar{\mathcal K}_j\cap U,\\
  +\infty,&\text{otherwise},
 \end{cases}
\end{equation}  
If
\begin{equation}\label{e:rate-regimes-intro}
 (1-s_j)|\log \eps_j|\longrightarrow\ell
 \qquad\text{for some }\ell\in[0,+\infty],
\end{equation}
then $(\bar{\mathcal F}_j)_j$ $\Gamma(L^p(U))$-converges to
\begin{equation}\label{e:local-scale-limit-functional-intro}
\bar{\mathcal F}^{(\ell)}(u):=
 \begin{cases}
 K_{n,p}\displaystyle\int_U|\nabla u|^p\,dx
 +\kappa_{\ell}K_{n,p}\capp(T)
  \displaystyle\int_U|u|^p\,dx,
 &u\in W^{1,p}(U),\\
 +\infty,&\text{otherwise},
 \end{cases}
\end{equation}
where
\begin{equation}\label{e:regime-coefficient-intro}
 \kappa_{\ell}:=
 \begin{cases}
  \displaystyle\exp\left(-\frac{np}{n-p}\ell\right),
       &\ell<+\infty,\\
  0,&\ell=+\infty.
 \end{cases}
\end{equation}
Thus the capacitary term is unchanged if $\ell=0$, is modulated by
$\exp(-np\ell/(n-p))$ if $\ell\in(0,+\infty)$, and disappears if $\ell=+\infty$.  

Note that the logarithmic parameter in \eqref{e:rate-regimes-intro} also governs the 
interaction between the BBM limit and periodic homogenization as established in \cite{BraidesMannellaSchiavoniPiazza}. 
In that quadratic setting, the limit is a convex combination of the homogenized energy and the energy with averaged coefficients, with weights $e^{-2\ell}$ and $1-e^{-2\ell}$, respectively. 
In formula \eqref{e:local-scale-limit-functional-intro} instead $\kappa_\ell$ affects only the capacitary coefficient, while the bulk term remains unchanged.

In the companion paper \cite{FocardiZeppieri} we will study the analogous problem in the stochastic setting following the approach in \cite{DeangelisFocardiZeppieri}. 
On the one hand results similar to Theorems~\ref{t:aperiodic-BBM} 
and~\ref{t:aperiodic-local-critical-scale} can be obtained, 
on the other hand several new additional phenomena 
intimately related to the stochastic setting arise.

\subsection{Key results for the proofs}
The proofs are inspired by the $\Gamma$-convergence analysis at fixed $s$ in \cite{FocardiAperiodic}, though several new ingredients are needed to deal with the simultaneous BBM localization and the homogenization of the obstacle sets.

To begin with, as the parameter of fractional differentiation $s$ tends to $1$ we need to revisit several well-known results and check their validity with uniform constants with respect to $s$  
itself (cf. Section~\ref{s:uniform-tools}). 

In addition, upon rescaling the fractional capacities by $1-s$, we establish the convergence of their values on bounded open sets to the corresponding value of the local capacity 
(cf. Corollary~\ref{c:capacity-convergence}). 
To this aim we revisit the BBM argument, and provide an extension of its $\Gamma$-convergence version with respect to the $L^p_{\mathrm{loc}}(\R^n)$ topology (Proposition~\ref{p:homogeneous-Gamma}). In Remark~\ref{r:capacity-thin-obstacles} the validity of such a convergence statement for general bounded sets is discussed at length.

The second key technical tool is an extension of the nonlocal joining lemma
of \cite{FocardiAperiodic} with estimates uniform as $s_j\uparrow1$ together with 
a further filling procedure (cf. Lemma~\ref{l:uniform-joining}).  
We recall that \cite[Lemmas~3.8--3.9]{FocardiAperiodic} build upon the ideas 
introduced by Ansini and Braides \cite{AnsiniBraides} for local 
energies on varying domains. The uniform joining lemma with filling in Lemma~\ref{l:uniform-joining} below shows 
that given a sequence $u_j\to u$ in $L^p(U)$ satisfying the obstacle
constraints and with bounded normalized fractional energy, the
joining construction modifies $u_j$ on annuli surrounding separated
obstacles.  The resulting function $w_j$ agrees with $u_j$ outside
these annuli and is constant on a smaller concentric annulus, with
value equal to the mean of $u_j$ over that annulus.  The modification
preserves the constraint on the enclosed obstacles and changes
the energy by an arbitrarily small relative error, uniformly
in $j$.  These constant annular values allow one to compare the
internal energies with the normalized relative fractional capacities of the
enclosed obstacles, multiplied by the $p$-th powers of the absolute
values of the corresponding means.  Summing these estimates and
passing to the limit yields the capacitary contribution.

We complement this construction with a filling procedure: inside
each ball enclosed by a constant annulus, we replace $w_j$ by the
same constant.  The filled sequence $v_j$ still converges to $u$
in $L^p(U)$.  The filling estimates separate the energy of $w_j$
into the full energy of $v_j$ and the internal energies accounting
for the obstacles, with an error tending to zero as $s_j\uparrow1$.
The full energy of $v_j$ then recovers the bulk Dirichlet term
through the BBM lower bound.  This separation is the additional
ingredient needed to obtain both contributions in the simultaneous
limit; its relation to the concentration mechanisms is explained next.

\subsection{Heuristic ideas of the asymptotics}\label{ss:heuristics}
Having described the main technical tools employed in what follows, we discuss next the mechanisms underlying the asymptotic analysis on an intuitive level.
As already mentioned above, both terms in \eqref{e:aperiodic-limit-functional-intro} arise
from concentration of the nonlocal energies near the diagonal.
For the gradient term this is the BBM limit,
whereas the capacitary term is obtained by summing the energies needed to impose the 
constraint near each obstacle.  Its coefficient is identified by the BBM limit of the fractional 
capacities after rescaling around the obstacles.
One might therefore expect different limits according to the relative speed at which
$\eps_j\downarrow0$ and $s_j\uparrow1$.  No such additional regime
occurs at the scale \eqref{e:obstacle-set-intro} (cf. Theorem~\ref{t:aperiodic-BBM}).
Instead, at the scale \eqref{e:local-critical-obstacle-scale-intro} such a trichotomy takes place
(cf. Theorem~\ref{t:aperiodic-local-critical-scale}) affecting the coefficient of the capacitray term.

Let us first explain why the coefficient of the gradient term is unchanged at both scales  \eqref{e:obstacle-set-intro} and \eqref{e:local-critical-obstacle-scale-intro}.
To this aim, we distinguish concentration around the obstacles from concentration
at small interaction distances.  Indeed, both the fractional and local scales $\lambda_j$,
$\bar{\lambda}_j$ are $o(\eps_j)$ as $j\to\infty$.
One can therefore surround the obstacles by regions much smaller
than the spacing of their centres but much larger than the
obstacles themselves. 
These regions occupy an overall vanishing volume, and for a fixed Lipschitz function 
carry a vanishing part of the normalized energy. 
The capacitary correctors can thus contribute a nonvanishing energy on regions carrying a vanishing part of the energy of the fixed target function. Removing these regions from the latter does not change its BBM limit.
The joining and filling estimates \eqref{e:filling-lower}--\eqref{e:filling-upper} in Lemma~\ref{l:uniform-joining} below give the corresponding separation 
for general sequences of bounded energy. 

The concentration of BBM energies near the diagonal involves a different scale.
If the lower cutoff for the interaction distance
is $\varepsilon_j$, then  the normalized radial
integral for an affine function satisfies, for every fixed
$\delta>0$ and all sufficiently large $j$,
\begin{equation}\label{e:BBM-mass-above-rj}
 (1-s_j)\int_{\eps_j}^{\delta}t^{p(1-s_j)-1}\,dt
 =\frac1p(\delta^{p(1-s_j)}-\eps_j^{p(1-s_j)}).
\end{equation}
If $\eps_j^{p(1-s_j)}\to a$, taking first $j\to\infty$ and then
$\delta\downarrow0$ gives the limit value $(1-a)/p$. 
Note that the energy contribution at distances greater than a fixed $\delta$ vanishes.
On the other hand, the BBM mass need not be concentrated at distances below $\eps_j$.
The parameter
\begin{equation}\label{e:relative-rate-parameter}
 \ell_j:=(1-s_j)|\log \eps_j|
\end{equation}
controls this radial distribution through
$\eps_j^{p(1-s_j)}=e^{-p\ell_j}$. It does not alter the spatial separation or the full bulk coefficient: the spatial regions
occupied by the capacitary corrections can be separated from
the energy of the limiting function without cutting the
interaction kernel at $\varepsilon_j$ as discussed above.

The capacitary coefficient requires a separate analysis.
It is independent of the relative rates
if the scaling of the obstacles is done according to the fractional 
critical choice in \eqref{e:obstacle-set-intro}, as shown by the identity
\begin{equation}\label{e:exact-critical-balance}
\lambda_j^{n-s_jp}=\eps_j^n
\end{equation}
valid for every $j$.  Indeed, by homogeneity of the fractional capacity,
the normalized cost of one obstacle is
\begin{equation*}
 (1-s_j)\operatorname{cap}_{s_j,p}(\lambda_jT)
 =(1-s_j)\eps_j^n\operatorname{cap}_{s_j,p}(T),
\end{equation*}
and Corollary~\ref{c:capacity-convergence} gives
\begin{equation*}
 (1-s_j)\operatorname{cap}_{s_j,p}(T)
 \longrightarrow K_{n,p}\capp(T).
\end{equation*}
Thus, as the number of obstacles per unit volume is of order $\eps_j^{-n}$, 
the capacitary coefficient contains no residual dependence on $\ell_j$.  
The exact balance \eqref{e:exact-critical-balance}
removes the rate dependence from the capacitary coefficient in
Theorem~\ref{t:aperiodic-BBM}.
The uniform joining and filling estimates separate this contribution from the 
bulk BBM energy without imposing any relation between $\eps_j$ and $s_j$.

Three regimes arise instead if one uses the
$s$-independent local critical scale
$\bar\lambda_j:=\eps_j^{n/(n-p)}$
instead of $\lambda_j$.  In that case the total normalized
capacitary cost contains the additional factor
\begin{equation}\label{e:local-scale-regime-factor}
 \eps_j^{-n}\bar\lambda_j^{n-s_jp}
 =\eps_j^{\frac{np(1-s_j)}{n-p}}
 =\exp\left(-\frac{np}{n-p}\ell_j\right).
\end{equation}
It converges to $1$, to $\exp(-np\ell/(n-p))$, or to $0$ according as
$\ell_j\to0$, $\ell_j\to \ell\in(0,+\infty)$, or $\ell_j\to+\infty$.
Thus spatial separation accounts for the retention of the full bulk
term, while \eqref{e:local-scale-regime-factor} gives the factor $\kappa_\ell$ in
\eqref{e:regime-coefficient-intro}.

\subsection{Outline of the paper}

Section~\ref{s:uniform-tools} contains the notation and the needed preliminaries 
on fractional Sobolev spaces, and the related uniform inequalities as $s\uparrow1$. 
Section~\ref{s:homogeneous-Gamma} establishes the convergence results for the normalized fractional capacities.  Section~\ref{s:uniform-joining-filling} is dedicated to the proof of the uniform joining and filling lemma. 
In Section~\ref{s:aperiodic-BBM}
 we state and prove the main results in the general aperiodic setting modelled 
 upon Delone sets of points. Appendix~\ref{s:singular-kernel-bounds} 
 collects some useful technical results used throughout the paper.

\section{Preliminaries and uniform inequalities as \texorpdfstring{$s\uparrow1$}{s tends to 1}}
\label{s:uniform-tools}

\subsection{Basic notation}
We denote by $\Ln(E)$ the Lebesgue measure of a measurable set $E$.
We use the notation $\chi_A$ for the characteristic function of a set $A$, and
\(
 (u)_A:=\fint_Au(x)\,dx
\)
for the mean value of an integrable function $u$ if $A$ is measurable
with $0<\Ln(A)<+\infty$.  We write $B_r(x):=\{y\in\Rn:|y-x|<r\}$ and denote its closure by
$\overline B_r(x)$.  We also write $B_r:=B_r(0)$,
$\omega_n:=\Ln(B_1)$. 
Equalities between functions are understood $\Ln$-almost everywhere 
(abbreviated as a.e.).

Standard Sobolev and fractional Sobolev spaces are denoted as usual.
In particular, throughout this paper $n\geq2$ and $p\in(1,n)$. 
For $s\in(0,1)$ with $sp<n$, set
\begin{equation}\label{e:critical-exponent}
 p_s^*:=\frac{np}{n-sp},
 \qquad
 [u]_{s,p}^p:=
 \int_{\Rn\times\Rn}
 \frac{|u(x)-u(y)|^p}{|x-y|^{n+sp}}\,dx\,dy.
\end{equation}
We denote by $\Dhom^{s,p}(\Rn)$ the completion of
$C_c^\infty(\Rn)$ with respect to $[\,\cdot\,]_{s,p}$.  By
\cite[Theorem~3.1]{BrascoGomezVazquez}, in the subconformal range
$sp<n$ this abstract completion is canonically identified with
\begin{equation}\label{e:homogeneous-characterization}
 \Dhom^{s,p}(\Rn)
 =\left\{u\in L^{p_s^*}(\Rn):[u]_{s,p}<+\infty\right\}.
\end{equation}
The same statement for $s=1$ gives
\begin{equation}\label{e:local-homogeneous-space}
 \Dhom^{1,p}(\Rn)
 =\left\{u\in L^{p^*}(\Rn):\nabla u\in L^p(\Rn;\Rn)\right\},
 \qquad p^*=\frac{np}{n-p},
\end{equation}
and $C_c^\infty(\Rn)$ is dense in this space for the norm
$\|\nabla u\|_{L^p(\Rn)}$.

\subsection{Uniform inequalities as \texorpdfstring{$s\uparrow1$}{s tends to 1}}
We first recall the form of the fractional Poincar\'e--Wirtinger
inequality which will be used below.  Let $A\subset\Rn$ be a bounded,
connected Lipschitz open set and let $O\subset A$ be measurable with
$\Ln(O)>0$.  For every $s_0\in(0,1)$ there exists
$c=c(n,p,s_0,A,O)>0$ such that
\begin{equation}\label{e:uniform-PW-prelim}
 \|u-(u)_O\|_{L^p(A)}^p
 \leq c(1-s)[u]_{s,p;A}^p
 \qquad
 \text{for every }s\in[s_0,1),\ u\in W^{s,p}(A).
\end{equation}
For cubes and the mean on the whole cube this is precisely
\cite[Corollary~2.1]{PoncePoincare}; the bounded Lipschitz-domain
version follows from \cite[Theorem~1.1]{PoncePoincare}, and the
passage from the mean on $A$ to the mean on $O$ follows from
H\"older's inequality.  By
scaling, if $z\in\Rn$ and $r>0$, then
\begin{equation}\label{e:uniform-PW-scaled}
 \|u-(u)_{z+rO}\|_{L^p(z+rA)}^p
 \leq c(1-s)r^{sp}[u]_{s,p;z+rA}^p.
\end{equation}
An elementary change of variables shows that the constant is uniform for families of pairs $(A,O)$ obtained from a
fixed pair by bi-Lipschitz transformations whose Lipschitz constants,
together with those of their inverses, are uniformly bounded.  This
is the uniform version of the estimate used in the joining lemma.

The second global estimate which will be crucial in what follows is the sharp-order fractional
Sobolev--Gagliardo--Nirenberg inequality of
Maz'ya and Shaposhnikova.  With the correction recorded in
\cite{MazyaShaposhnikovaErratum},
\cite[Theorem~1]{MazyaShaposhnikova} states that
\begin{equation}\label{e:MS-inequality}
 \|u\|_{L^{p_s^*}(\Rn)}^p
 \leq c(n,p)\frac{s(1-s)}{(n-sp)^{p-1}}[u]_{s,p}^p,
 \qquad u\in\Dhom^{s,p}(\Rn).
\end{equation}
In particular, if $s\in[s_0,1)$ and $p\in(1,n)$, then
\begin{equation}\label{e:uniform-Sobolev}
 \|u\|_{L^{p_s^*}(\Rn)}^p
 \leq c(n,p,s_0)(1-s)[u]_{s,p}^p.
\end{equation}
Unlike \eqref{e:uniform-PW-prelim}, estimate
\eqref{e:uniform-Sobolev} fixes the additive constant at infinity and
is therefore the coercive estimate appropriate to homogeneous spaces.

We finally set
\begin{equation}\label{e:Knp}
 K_{n,p}:=\frac1p\int_{\mathbb S^{n-1}}|e\cdot\sigma|^p\,
 d\mathcal H^{n-1}(\sigma),
\end{equation}
where $e\in\mathbb S^{n-1}$ is arbitrary.  With the convention
\eqref{e:critical-exponent}, the Bourgain--Brezis--Mironescu formula
\cite{BBM} reads
\begin{equation}\label{e:BBM-formula}
 \lim_{s\uparrow1}(1-s)[u]_{s,p}^p
 =K_{n,p}\int_{\Rn}|\nabla u|^p\,dx
 \qquad\text{for }u\in W^{1,p}(\Rn).
\end{equation}

\section{Homogeneous \texorpdfstring{$\Gamma$}{Gamma}-convergence and capacities}
\label{s:homogeneous-Gamma}

In this section we prove the convergence of the rescaled fractional
capacities of bounded open sets to the corresponding local capacity.
For this purpose, we need a BBM formula in homogeneous spaces with
respect to the $L^p_{\mathrm{loc}}(\Rn)$ topology.  The results
available in the literature concern related, but different, settings:
the Gagliardo seminorm on a bounded Lipschitz open set, with the
$L^p(U)$ topology, in \cite[Theorem~8]{PonceGamma}; whole-space
energies, with the $L^p(\Rn)$ topology, in
\cite[Theorem~3.1]{GennaioliStefani}; and energies with homogeneous
Dirichlet conditions on a bounded Lipschitz open set, again with the
global $L^p$ topology, in
\cite[Theorem~2.1]{CrismaleDeLucaKubinNinnoPonsiglione} for $p=2$
and in \cite[Theorem~1.1]{KubinSaraccoStefani} for general $p$ and
weighted kernels.
For our purposes we need convergence in the local topology together with control at infinity. 
The following proposition provides this homogeneous version of the BBM result.
\begin{proposition}[Homogeneous BBM $\Gamma$-convergence]
\label{p:homogeneous-Gamma}
Let $s_j\uparrow1$ and define $F_j:L^p_{\mathrm{loc}}(\Rn)\to
[0,+\infty]$ by
\begin{equation}\label{e:Fj-homogeneous}
 F_j(u):=
 \begin{cases}
  (1-s_j)[u]_{s_j,p}^p,&u\in\Dhom^{s_j,p}(\Rn),\\
  +\infty,&\text{otherwise}.
 \end{cases}
\end{equation}
Then
\begin{equation}\label{e:homogeneous-Gamma-limit}
 \Gamma(L^p_{\mathrm{loc}}(\Rn))\hbox{-}
 \lim_{j\to+\infty}F_j=F,
\end{equation}
where
\begin{equation}\label{e:F-homogeneous}
 F(u):=
 \begin{cases}
  K_{n,p}\displaystyle\int_{\Rn}|\nabla u|^p\,dx,
       &u\in\Dhom^{1,p}(\Rn),\\
  +\infty,&\text{otherwise}.
 \end{cases}
\end{equation}
Moreover, $(F_j)_j$ is equicoercive in
$L^p_{\mathrm{loc}}(\Rn)$: every sequence $(u_j)_j$ such that
\[
 \sup_jF_j(u_j)<+\infty
\]
is relatively compact in
$L^p_{\mathrm{loc}}(\Rn)$.
\end{proposition}

\begin{proof}
We start with compactness.  By \eqref{e:uniform-Sobolev}, whenever $ \sup_jF_j(u_j)<+\infty$
\begin{equation}\label{e:global-coercivity}
 \sup_j\|u_j\|_{L^{p_{s_j}^*}(\Rn)}<+\infty.
\end{equation}
Consequently $(u_j)_j$ is bounded in $L^p(B_R)$ for every $R>0$.
The local compactness theorem of Bourgain--Brezis--Mironescu, in the
form proved in \cite[Theorem~1.2]{PoncePoincare}, applied on $B_R$ and
followed by a diagonal argument, gives relative compactness in
$L^p_{\mathrm{loc}}(\Rn)$.

Let now $u_j\to u$ in $L^p_{\mathrm{loc}}(\Rn)$ and suppose, without
loss of generality, that $\sup_jF_j(u_j)<+\infty$.  For every $R>0$,
the local $\Gamma$-liminf inequality in
\cite[Theorem~8]{PonceGamma} gives
\begin{equation}\label{e:local-liminf}
 \liminf_{j\to+\infty}F_j(u_j)
 \geq K_{n,p}\int_{B_R}|\nabla u|^p\,dx.
\end{equation}
Hence $u\in W^{1,p}_{\mathrm{loc}}(\Rn)$ and, by letting
$R\to+\infty$,
\begin{equation}\label{e:global-liminf-gradient}
 \liminf_{j\to+\infty}F_j(u_j)
 \geq K_{n,p}\int_{\Rn}|\nabla u|^p\,dx.
\end{equation}
It remains to identify the behaviour at infinity.  Up to a
subsequence, $u_j\to u$ almost everywhere in $\Rn$; since
$p_{s_j}^*\to p^*$, Fatou's lemma and \eqref{e:global-coercivity}
imply
\begin{equation}\label{e:limit-Lpstar}
 \int_{\Rn}|u|^{p^*}\,dx
 \leq\liminf_{j\to+\infty}
 \int_{\Rn}|u_j|^{p_{s_j}^*}\,dx<+\infty.
\end{equation}
Thus $u\in\Dhom^{1,p}(\Rn)$ by
\eqref{e:local-homogeneous-space}, and
\eqref{e:global-liminf-gradient} is the required $\Gamma$-liminf
inequality.

To prove the $\Gamma$-limsup inequality, let
$u\in\Dhom^{1,p}(\Rn)$.  By
\cite[Theorem~3.1]{BrascoGomezVazquez}, there exist
$\varphi_k\in C_c^\infty(\Rn)$ such that
\begin{equation}\label{e:homogeneous-density}
 \varphi_k\to u\quad\text{in }L^{p^*}(\Rn),
 \qquad
 \nabla\varphi_k\to\nabla u\quad\text{in }L^p(\Rn;\Rn).
\end{equation}
For every fixed $k$, the BBM formula yields
\begin{equation}\label{e:BBM-fixed-k}
 \lim_{j\to+\infty}F_j(\varphi_k)
 =K_{n,p}\int_{\Rn}|\nabla\varphi_k|^p\,dx.
\end{equation}
A standard diagonal choice $k=k(j)\uparrow+\infty$ in
\eqref{e:homogeneous-density}--\eqref{e:BBM-fixed-k} provides a
sequence converging to $u$ in $L^p_{\mathrm{loc}}(\Rn)$ and proves
the $\Gamma$-limsup inequality.
\end{proof}

We record the capacitary consequence needed below.  If
$T\subset\Rn$ is a bounded open set, define
\begin{equation}\label{e:fractional-capacity}
 \capsp(T):=\inf\left\{[u]_{s,p}^p:
 u\in\Dhom^{s,p}(\Rn),\ u\geq1\text{ a.e. on }T\right\}
\end{equation}
and
\begin{equation}\label{e:local-capacity}
 \capp(T):=\inf\left\{\int_{\Rn}|\nabla u|^p\,dx:
 u\in\Dhom^{1,p}(\Rn),\ u\geq1\text{ a.e. on }T\right\}.
\end{equation}
If $T\Subset B_R$, the corresponding relative capacities are
\begin{equation}\label{e:fractional-relative-capacity}
 \begin{aligned}
 \Capsp(T,B_R):=\inf\bigl\{[u]_{s,p}^p:\;&
 u\in\Dhom^{s,p}(\Rn),\ u=0\text{ a.e. on }B_R^c,\\[-2pt]
 &u\geq1\text{ a.e. on }T\bigr\}.
 \end{aligned}
\end{equation}
\begin{equation}\label{e:local-relative-capacity}
 \begin{aligned}
 \Capp(T,B_R):=\inf\biggl\{\int_{\Rn}|\nabla u|^p\,dx:\;&
 u\in\Dhom^{1,p}(\Rn),\ u=0\text{ a.e. on }B_R^c,\\[-2pt]
 &u\geq1\text{ a.e. on }T\biggr\}.
 \end{aligned}
\end{equation}

\begin{corollary}[Convergence of the minimum problems]
\label{c:capacity-convergence}
Let $T\subset\Rn$ be a bounded open set and let $s_j\uparrow1$.
Then
\begin{equation}\label{e:global-capacity-limit}
 \lim_{j\to+\infty}(1-s_j)\operatorname{cap}_{s_j,p}(T)
 =K_{n,p}\capp(T).
\end{equation}
If $T\Subset B_R$, then
\begin{equation}\label{e:relative-capacity-limit}
 \lim_{j\to+\infty}(1-s_j)C_{s_j,p}(T,B_R)
 =K_{n,p}\Capp(T,B_R).
\end{equation}
Finally, if $L_j\to+\infty$ and $T\Subset B_{L_j}$ eventually, then
\begin{equation}\label{e:expanding-capacity-limit}
 \lim_{j\to+\infty}(1-s_j)C_{s_j,p}(T,B_{L_j})
 =K_{n,p}\capp(T).
\end{equation}
\end{corollary}

\begin{proof}
Formula~\eqref{e:global-capacity-limit} follows from
Proposition~\ref{p:homogeneous-Gamma} and the fundamental theorem of $\Gamma$-convergence (cf. \cite[Theorem~7.8]{DalMaso-book}), applied to the functionals $F_j$, $F$ restricted to the closed set
\[
 \mathcal A_T:=\{v\in L^p_{\mathrm{loc}}(\Rn):
                         v\geq1\text{ a.e. on }T\}.
\]
Equicoercivity and the lower bound are inherited from
Proposition~\ref{p:homogeneous-Gamma}.  To check that recovery
sequences can be chosen in $\mathcal A_T$, let
$v\in\mathcal A_T\cap\Dhom^{1,p}(\Rn)$ and take
$v_R=\eta_Rv$, where $\eta_R\in C_c^\infty(B_{2R})$,
$0\leq\eta_R\leq1$, $\eta_R=1$ on $B_R$,
$|\nabla\eta_R|\leq c/R$, and $T\Subset B_R$.  Then
$v_R\in\mathcal A_T\cap W^{1,p}(\Rn)$, and
\begin{equation}\label{e:capacity-cut-off}
 \int_{\Rn}|v|^p|\nabla\eta_R|^p\,dx
 \leq c\left(\int_{B_{2R}\setminus B_R}|v|^{p^*}\,dx\right)^{p/p^*}
 \longrightarrow0.
\end{equation}
Thus $\nabla v_R\to\nabla v$ in $L^p(\Rn;\Rn)$.  The BBM formula
for fixed $R$, followed by a diagonal choice $R=R(j)\to+\infty$,
gives the required recovery sequence.  The convergence of the
minimum values now follows from the cited theorem.

A similar argument works for the relative capacities.
We give the details  for the reader's convenience. 

Fix $R$ with $T\Subset B_R$ and choose admissible $u_j$ such that
\[
 (1-s_j)[u_j]_{s_j,p}^p
 \leq (1-s_j)C_{s_j,p}(T,B_R)+\frac1j.
\]
By truncation we may assume $0\leq u_j\leq1$.  A fixed function
in $C_c^\infty(B_R)$ equal to one near $\overline T$, together
with the BBM formula, bounds the right-hand side uniformly in $j$.
Along a subsequence realizing the lower limit,
Proposition~\ref{p:homogeneous-Gamma} therefore gives
$u_j\to u$ in $L^p_{\mathrm{loc}}(\Rn)$, with
$u\in\Dhom^{1,p}(\Rn)$.  Since $u_j=1$ a.e. on $T$ and
$u_j=0$ a.e. on $B_R^c$, the same conditions hold for $u$.
The lower bound in that proposition yields
\[
 \liminf_{j\to+\infty}(1-s_j)C_{s_j,p}(T,B_R)
 \geq K_{n,p}\int_{\Rn}|\nabla u|^p\,dx
 \geq K_{n,p}\Capp(T,B_R).
\]
Conversely, every $v$ admissible for $\Capp(T,B_R)$ belongs to
$W^{1,p}(\Rn)$, since it has bounded support and belongs to
$L^{p^*}(\Rn)$.  It is therefore admissible for
$C_{s_j,p}(T,B_R)$ for every $j$, and
\[
 \limsup_{j\to+\infty}(1-s_j)C_{s_j,p}(T,B_R)
 \leq\lim_{j\to+\infty}(1-s_j)[v]_{s_j,p}^p
 =K_{n,p}\int_{\Rn}|\nabla v|^p\,dx.
\]
Taking the infimum over $v$ proves
\eqref{e:relative-capacity-limit}.

Finally, fix $R$ with $T\Subset B_R$.  For all sufficiently
large $j$ we have $L_j\geq R$, and
monotonicity with respect to the outer domain gives
\begin{equation}\label{e:capacity-sandwich}
 \operatorname{cap}_{s_j,p}(T)
 \leq C_{s_j,p}(T,B_{L_j})
 \leq C_{s_j,p}(T,B_R).
\end{equation}
First letting $j\to+\infty$ and then $R\to+\infty$, and using
$\Capp(T,B_R)\downarrow\capp(T)$ (which follows from the same
cut-off argument as in \eqref{e:capacity-cut-off}), proves
\eqref{e:expanding-capacity-limit}.
\end{proof}

\begin{remark}[Capacity of thin obstacles]\label{r:capacity-thin-obstacles}
The convergence in \eqref{e:global-capacity-limit} does not hold for every bounded set. Recall that capacity is extended by outer regularity to general bounded sets $T\subset\R^n$, i.e.
\[
 \operatorname{cap}_{s,p}(T)
 :=\inf\{\operatorname{cap}_{s,p}(O):T\subset O,
                                  \ O\text{ bounded and open}\}.
\]
Outer regularity and Corollary~\ref{c:capacity-convergence} give only
\begin{equation}\label{e:capacity-general-limsup}
 \limsup_{j\to\infty}(1-s_j)\operatorname{cap}_{s_j,p}(T)
 \leq K_{n,p}\capp(T).
\end{equation}
Indeed, the inequality may be strict: critical Cantor-type compact sets can have Hausdorff dimension $n-p$ and
positive $p$-capacity, while their fractional capacities vanish for every $s\in(0,1)$ as $n-p<n-sp$ 
(see \cite[Section~5.3 and Theorem~5.1.13]{AdamsHedberg}).

For $p=2$, the upper bound for compact sets and the full limit for
closures of bounded Lipschitz open sets are proved in
\cite[Proposition~1.2]{CintiOgnibeneRuffini}.

On the other hand, using the monotonicity property of capacities under set inclusion, the equality in \eqref{e:global-capacity-limit}
can be established more generally if $\capp(T)=\capp(\mathrm{int}T)$. 
\end{remark}

Following \cite[p.~3510]{FocardiAperiodic}, for $m\geq2$, $R>1$,
and $T\Subset B_R$, we also introduce the interior relative capacity
\begin{equation}\label{e:interior-relative-capacity-new}
 \begin{aligned}
 \operatorname{cap}_{s,p}(T,B_{mR};R)
 :=\inf\bigl\{[\varphi]_{s,p;B_{mR}}^p:\;&
 \varphi\in\Dhom^{s,p}(\Rn),\quad
 \varphi=0\text{ a.e. on }\Rn\setminus B_R,\\[-2pt]
 &\varphi\geq1\quad\text{a.e. on }T\bigr\}.
 \end{aligned}
\end{equation}
The following lemma gives the uniform-in-$s$ comparison of the
interior and global capacities, together with the complete limit
for expanding domains.  It is the counterpart of
\cite[Lemma~2.12]{FocardiAperiodic} in the present normalization.
\begin{lemma}[Uniform comparison and convergence of the interior capacity]
\label{l:interior-capacity-BBM}
Let $T\subset\Rn$ be bounded and open and let $m\geq2$ be fixed.
There is a constant $c_m=c(n,p,s_0,m)>0$ such that, for every
$s\in[s_0,1)$ and $R>1$ with $T\Subset B_R$,
\begin{equation}\label{e:interior-global-comparison-new}
 \frac{\capsp(T)}{1+c_m(1-s)}
 \leq\operatorname{cap}_{s,p}(T,B_{mR};R)
 \leq C_{s,p}(T,B_R).
\end{equation}
Consequently, if $L_j\to+\infty$ and $s_j\uparrow1$, then
\begin{equation}\label{e:interior-capacity-limit-new}
 \lim_{j\to+\infty}(1-s_j)
 \operatorname{cap}_{s_j,p}(T,B_{mL_j};L_j)
 =K_{n,p}\capp(T).
\end{equation}
\end{lemma}

\begin{proof}
Let $\varphi$ be admissible in
\eqref{e:interior-relative-capacity-new}.  Since $\varphi=0$ on
$B_{mR}\setminus B_R$, the scaled uniform
Poincar\'e--Wirtinger inequality \eqref{e:uniform-PW-scaled}, applied
to the fixed pair $(B_m,B_m\setminus B_1)$, gives
\begin{equation}\label{e:interior-capacity-PW-new}
 \|\varphi\|_{L^p(B_{mR})}^p
 \leq c_m(1-s)R^{sp}[\varphi]_{s,p;B_{mR}}^p.
\end{equation}
A direct integration of the tail, as in (ii) of
Lemma~\ref{l:singular-kernel-bounds}, and \eqref{e:interior-capacity-PW-new} therefore yield
\begin{align}
 [\varphi]_{s,p}^p
 &= [\varphi]_{s,p;B_{mR}}^p
 +2\int_{B_R}\int_{B_{mR}^c}
 \frac{|\varphi(x)|^p}{|x-y|^{n+sp}}\,dy\,dx\notag\\
 &\leq[\varphi]_{s,p;B_{mR}}^p
 +\frac{2n\omega_n}{sp}((m-1)R)^{-sp}\|\varphi\|_{L^p(B_R)}^p\notag\\
 &\leq\bigl(1+c_m(1-s)\bigr)
 [\varphi]_{s,p;B_{mR}}^p.
 \label{e:interior-capacity-tail-new}
\end{align}
Taking the infimum proves the first inequality in
\eqref{e:interior-global-comparison-new}.  The second follows by
restricting the energy of each competitor for $C_{s,p}(T,B_R)$
to $B_{mR}\times B_{mR}$.  Multiplying the resulting inequalities
by $1-s_j$, with $R=L_j$, and applying
Corollary~\ref{c:capacity-convergence} proves
\eqref{e:interior-capacity-limit-new}.
\end{proof}
\begin{remark}
The use of the scaled uniform Poincar\'e--Wirtinger inequality leading to \eqref{e:interior-capacity-PW-new} is not necessary.
Indeed, let $C_{n,p}$ denote the constant $c(n,p)$ in \eqref{e:MS-inequality}.
Since $\varphi=0$ outside $B_R$, H\"older's inequality and \eqref{e:MS-inequality} give
\[
 \|\varphi\|_{L^p(B_R)}^p
 \leq \omega_n^{sp/n}R^{sp}\|\varphi\|_{L^{p_s^*}(\Rn)}^p
 \leq C_{n,p}\omega_n^{sp/n}R^{sp}
       \frac{s(1-s)}{(n-sp)^{p-1}}[\varphi]_{s,p}^p.
\]
Inserting this estimate in the second line of \eqref{e:interior-capacity-tail-new} yields
\[
 \left(1-
 \frac{2nC_{n,p}\omega_n^{1+sp/n}}{p(n-sp)^{p-1}(m-1)^{sp}}
 (1-s)\right)[\varphi]_{s,p}^p
 \leq [\varphi]_{s,p;B_{mR}}^p.
\]
Having fixed $p\in(1,n)$, the coefficient multiplying $1-s$ is uniformly bounded 
for $s\in(0,1)$ and $m\geq2$.
Thus the argument above gives the same asymptotically exact comparison.
\end{remark}

Finally, we point out that the gain with respect to the fixed-order comparison is that the relative error in \eqref{e:interior-global-comparison-new} tends
to zero while $m$ remains fixed.  In
\cite[Lemma~2.12]{FocardiAperiodic}, the corresponding conclusion
is obtained by letting the ratio of the outer and inner radii
diverge.  Here the uniform Poincar\'e inequality supplies the
additional factor $1-s_j$, so a fixed ratio $m$ suffices and no
relation between $L_j$ and $s_j$ is needed.  This is a consequence
of the relative tail estimate, rather than merely of multiplying
a fixed-order capacity estimate by $1-s_j$.

\section{The uniform joining lemma with filling}
\label{s:uniform-joining-filling}

In this section we establish the main technical tool to prove the
asymptotics of the rescaled Gagliardo seminorms. More precisely,
we prove a joining and filling lemma for sequences with bounded
rescaled fractional energy.  The construction follows
\cite[Lemmas~3.8--3.9]{FocardiAperiodic}, up to checking that the relevant constants there are uniform
as $s_j\uparrow1$.  Only separation of the centres is needed to
make the functions constant on suitable annuli.  Filling the
enclosed balls with the same constants then separates the energy
of the filled sequence from the internal energies of the original
one, with an error controlled by the ratio of the two radii.
These estimates are collected in Lemma~\ref{l:uniform-joining}.
The conclusion concerning averages on Vorono\"i cells related to Delone sets of points 
is stated separately in Corollary~\ref{c:Delone-joining}.

In this section only, the restriction $p<n$ is dropped.  Let
$U\subset\Rn$ be a bounded open set, let $p\in(1,+\infty)$, and
let $s_j\uparrow1$.  
After discarding finitely many terms, fix
$s_0\in(\max\{0,1-n/p\},1)$ such that $s_j\in[s_0,1)$.
For $u\in L^p(U)$ and every measurable set
$E\subseteq U\times U$ define the normalized interaction energy by
\begin{equation}\label{e:defectj}
 \defectj(u,E):=(1-s_j)
 \int_E\frac{|u(x)-u(y)|^p}{|x-y|^{n+s_jp}}\,dx\,dy.
\end{equation}
For every $j\in\N$, let $\Ieps$ be a finite index set, let
$\{\xiij\}_{\ii\in\Ieps}\subset U$, and let
$\rj>0$ be such that
\begin{equation}\label{e:separation}
 |\xiij-x_j^{\kk}|\geq2\rj\quad\text{if }\ii\neq\kk,
 \qquad
 B_{\rj}(\xiij)\subset U\quad\text{for every }\ii\in\Ieps.
\end{equation}
We assume throughout that $\rj\to0$.  With fixed $m\in\N$, $m\geq2$,
for integers $h\geq1$ and $\ii\in\Ieps$ set, as in
\cite[Lemma~3.9]{FocardiAperiodic},
\begin{equation}\label{e:joining-sets}
 \Bijh:=B_{m^{-3h}\rj}(\xiij),
 \qquad
 \Cijh:=B_{m^{-3h-1}\rj}(\xiij)
 \setminus\overline{B}_{m^{-3h-2}\rj}(\xiij).
\end{equation}
We complement the joining lemma of \cite[Lemma~3.9]{FocardiAperiodic} needed to extract
the underlying capacitary contribution on sequences $(u_j)_j$ with equi-bounded 
$\defectj(u_j,U\times U)$ energies, with a filling procedure to recover the BBM localization.
In addition, on those sequences, filling separates the two contributions with an error tending to zero 
without sending $m$ to infinity.
\begin{lemma}[Uniform joining lemma with filling]\label{l:uniform-joining}
Let $u_j\to u$ strongly in $L^p(U)$ and assume that
\begin{equation}\label{e:energy-bound}
 \sup_j\defectj(u_j,U\times U)<+\infty.
\end{equation}
For every fixed pair of integers $m\geq2$, $N\geq1$, there exist
$h_j\in\{1,\ldots,N\}$ and $w_j\in W^{s_j,p}(U)$ such that
\begin{equation}\label{e:equal1}
 w_j=u_j\quad\text{on }U\setminus
 \bigcup_{\ii\in\Ieps}
 \bigl(\overline B_j^{\ii,h_j}\setminus B_j^{\ii,h_j+1}\bigr),
\end{equation}
\begin{equation}\label{e:equal2}
 w_j=(u_j)_{\Cij}\quad\text{on }\Cij,
\end{equation}
and, for every measurable set $E\subseteq U\times U$,
\begin{equation}\label{e:uniform-joining-error}
 \left|\defectj(u_j,E)-\defectj(w_j,E)\right|
 \leq\frac{c_m}{N}\defectj(u_j,U\times U).
\end{equation}
Here and below $c_m=c(n,p,s_0,m)>0$ is independent of $j$ and $N$.
Moreover,
\begin{equation}\label{e:wj-convergence}
 w_j\longrightarrow u\qquad\text{strongly in }L^p(U).
\end{equation}

For the filling construction, let
\begin{equation}\label{e:filling-balls}
 \rho_j:=m^{-3h_j-2}\rj,
 \qquad
 D_j^{\ii}:=B_{\rho_j}(\xiij),
 \qquad
 P_j^{\ii}:=B_{m\rho_j}(\xiij).
\end{equation}
Then $\Cij=P_j^{\ii}\setminus\overline{D_j^{\ii}}$.  Define the filled sequence by
\begin{equation}\label{e:vj}
 v_j(x):=
 \begin{cases}
  (u_j)_{\Cij},&x\in D_j^{\ii},\quad\ii\in\Ieps,\\
  w_j(x),&x\in U\setminus\displaystyle\bigcup_{\ii\in\Ieps}D_j^{\ii}.
 \end{cases}
\end{equation}
Then $v_j\in W^{s_j,p}(U)$,
\begin{equation}\label{e:vj-convergence}
 v_j\longrightarrow u\qquad\text{strongly in }L^p(U),
\end{equation}
and the filling satisfies the following estimates.  For every
measurable $E\subseteq U\times U$, set
\begin{equation}\label{e:SjE}
 \mathcal S_j(E):=\sum_{\ii\in\Ieps}
 \defectj\bigl(w_j,E\cap(P_j^{\ii}\times P_j^{\ii})\bigr),
 \qquad\mathcal S_j:=\mathcal S_j(U\times U).
\end{equation}
There is a constant $c=c(n,p,s_0)>0$, independent of $j,m,N$,
such that, on putting
\begin{equation}\label{e:filling-error-factor}
 \delta_{j,m}:=c(1-s_j)^{1/p}m^{-s_j},
\end{equation}
we have
\begin{equation}\label{e:filling-difference}
 \left|\defectj(w_j,E)-\defectj(v_j,E)-\mathcal S_j(E)\right|
 \leq\delta_{j,m}
       \bigl(\defectj(v_j,U\times U)+\mathcal S_j\bigr).
\end{equation}
In particular,
\begin{equation}\label{e:filling-lower}
 \defectj(w_j,U\times U)
 \geq(1-\delta_{j,m})
       \bigl(\defectj(v_j,U\times U)+\mathcal S_j\bigr)
\end{equation}
and
\begin{equation}\label{e:filling-upper}
 \defectj(w_j,U\times U)
 \leq(1+\delta_{j,m})
       \bigl(\defectj(v_j,U\times U)+\mathcal S_j\bigr).
\end{equation}
For each fixed $m\geq2$, $\delta_{j,m}\to0$ as $j\to\infty$.

Finally, if $u_j\in L^\infty(U)$, then
\begin{equation}\label{e:Linfty}
 \|w_j\|_{L^\infty(U)}\leq\|u_j\|_{L^\infty(U)},
 \qquad
 \|v_j\|_{L^\infty(U)}\leq\|u_j\|_{L^\infty(U)}.
\end{equation}
\end{lemma}

\begin{proof}
Given $m,N\in\N$, for every $j\in\N$ and
$h\in\{1,\ldots,N\}$ fixed, let
\begin{equation}\label{e:Ah}
 A_j^{\ii,h}:=\Bijh\setminus\overline B_j^{\ii,h+1},
 \qquad
 A_j^h:=\bigcup_{\ii\in\Ieps}A_j^{\ii,h},
 \qquad
 z_j^{\ii,h}:=(u_j)_{\Cijh}.
\end{equation}
Put $\rho_{j,h}:=m^{-3h}\rj$.  For every $\ii\in\Ieps$, choose
$\varphi_j^{\ii,h}\in W^{1,\infty}(\Rn)$ such that
\begin{equation}\label{e:joining-cut-off}
 \begin{gathered}
  0\leq\varphi_j^{\ii,h}\leq1,
  \qquad
  \varphi_j^{\ii,h}=1\quad\text{on }\Cijh,
  \qquad
  \varphi_j^{\ii,h}=0\quad\text{on }
  \Rn\setminus A_j^{\ii,h},\\
  \operatorname{Lip}(\varphi_j^{\ii,h})
  \leq\frac{2m^2}{\rho_{j,h}}.
 \end{gathered}
\end{equation}
Such a cut-off exists because $\Cijh$ is separated from both
components of $\partial A_j^{\ii,h}$ by a distance bounded from
below by $\rho_{j,h}/(2m^2)$.  Since the sets
$A_j^{\ii,h}$ are mutually disjoint, the joining modification is
given explicitly by
\begin{align}
 w_j^h(x)
 &:={\sum}_{\ii\in\Ieps}\varphi_j^{\ii,h}(x)z_j^{\ii,h}
 +\left(1-{\sum}_{\ii\in\Ieps}\varphi_j^{\ii,h}(x)\right)u_j(x)
 \notag\\
 &=u_j(x)+{\sum}_{\ii\in\Ieps}\varphi_j^{\ii,h}(x)
 \bigl(z_j^{\ii,h}-u_j(x)\bigr).
 \label{e:wjh-definition}
\end{align}
In particular, $w_j^h=z_j^{\ii,h}$ on $\Cijh$ and
$w_j^h=u_j$ on $U\setminus A_j^h$.

Apply \cite[Lemma~3.8]{FocardiAperiodic} with
$(A')_{\ii}^h=\Cijh$, $A_{\ii}^h$ as in \eqref{e:Ah}, and
$z_{\ii}=z_j^{\ii,h}$.  The radius $\rho$ used in that lemma is
$\rho_{j,h}=m^{-3h}\rj$, and its resulting function is precisely
\eqref{e:wjh-definition}.  The normalized form of
\cite[Lemma~3.8, formula~(3.8)]{FocardiAperiodic} gives, for every measurable set
$E\subseteq U\times U$,
\begin{multline}\label{e:technical-uniform}
 \left|\defectj(u_j,E)-\defectj(w_j^h,E)\right|\\
 \leq c\defectj(u_j,U\times A_j^h)
 +cm^{2p}\left(\frac{m^{3h}}{\rj}\right)^{s_jp}
 \sum_{\ii\in\Ieps}
 \int_{A_j^{\ii,h}}|u_j-z_j^{\ii,h}|^p\,dx,
\end{multline}
where $c=c(n,p,s_0)>0$ is independent of $j$ and $h$.

Since the constant in \cite[Lemma~3.8]{FocardiAperiodic} is stated
there for fixed $s$, we verify that its normalized version is
uniform as $s\uparrow1$.  Fix $j$ and $h$, and abbreviate
\[
 s:=s_j,\qquad \rho:=\rho_{j,h},\qquad
 A_{\ii}:=A_j^{\ii,h},\qquad z_{\ii}:=z_j^{\ii,h},
 \qquad \varphi_{\ii}:=\varphi_j^{\ii,h}.
\]
By \eqref{e:joining-cut-off}, $0\leq\varphi_{\ii}\leq1$,
$\operatorname{Lip}(\varphi_{\ii})\leq2m^2/\rho$, and the supports
of the cut-offs are contained in the mutually disjoint sets
$A_{\ii}$.  In the proof of \cite[Lemma~3.8]{FocardiAperiodic}, the
dependence on $s$ occurs only in the singular-kernel estimates in
formulas (3.11), (3.14)--(3.15), and (3.17)--(3.18); the constants
coming from the elementary inequalities for the $p$-power are
independent of $s$.

We first consider the interactions inside a single $A_{\ii}$.  The
term produced by the cut-off is bounded, after multiplication by
$1-s$, by
\begin{align}
 &(1-s)\left(\frac{2m^2}{\rho}\right)^p
 \sum_{\ii\in\Ieps}\int_{A_{\ii}}|u_j(y)-z_{\ii}|^p
 \left(\int_{A_{\ii}}
 \frac{dx}{|x-y|^{n-(1-s)p}}\right)dy.                 \label{e:single-annulus-cut-off}
\end{align}
Here $\nu=n-(1-s)p$ belongs to $(0,n)$ by the choice of $s_0$.
Since $\Ln(A_{\ii})\leq c_n\rho^n$, item~(i) of
Lemma~\ref{l:singular-kernel-bounds} gives
\begin{equation}\label{e:uniform-near-kernel}
 (1-s)\int_{A_{\ii}}
 \frac{dx}{|x-y|^{n-(1-s)p}}
 \leq c(n,p)\rho^{p(1-s)}.
\end{equation}
Indeed, the factor $(n-\nu)^{-1}=(p(1-s))^{-1}$ is exactly
cancelled by the normalization $1-s$.  Consequently,
\eqref{e:single-annulus-cut-off} is at most
\begin{equation}\label{e:single-annulus-uniform}
 c(n,p)m^{2p}\rho^{-sp}
 \sum_{\ii\in\Ieps}\int_{A_{\ii}}|u_j-z_{\ii}|^p\,dx.
\end{equation}
This is the uniform version of formula (3.11) in the cited proof.

We next consider the interactions between distinct annuli.  The
argument leading to \cite[formula~(3.14)]{FocardiAperiodic} uses
only the separation of the centres.  More precisely, the packing
estimate
\[
 \#\bigl\{\kk\in\Ieps:\ell\rj<|x_j^{\ii}-x_j^{\kk}|
 \leq(\ell+1)\rj\bigr\}\leq c_n\ell^{n-1}
\]
holds for every integer $\ell\geq1$ and every finite $2\rj$-separated
family; no relative-density assumption is needed. The shells with $\ell\geq1$ then contain all the other centres.

Hence, for $x\in A_{\ii}$,
\begin{equation}\label{e:distinct-annuli-uniform}
 \begin{aligned}
 \sum_{\kk\ne\ii}\int_{A_{\kk}}
 \frac{dy}{|x-y|^{n+sp}}
 &\leq c(n,p)\frac{\rho^n}{\rj^{n+sp}}
 \sum_{\ell=1}^{+\infty}\ell^{-1-sp}\\
 &\leq c(n,p)\rho^{-sp}\sum_{\ell=1}^{+\infty}
 \ell^{-1-sp}
 \leq c(n,p,s_0)\rho^{-sp}.
 \end{aligned}
\end{equation}
In the second inequality we used $\rho=m^{-3h}\rj\leq\rj$.
Multiplication by $1-s\leq1$ preserves this bound.  Thus the
constant in formulas (3.14)--(3.15) is uniform for
$s\in[s_0,1)$.

It remains to check the locality-defect term, corresponding to
formulas (3.16)--(3.19) of the cited proof.  For
$|x-y|<\rho$, the Lipschitz estimate for $\varphi_{\ii}$ and
item~(i) of Lemma~\ref{l:singular-kernel-bounds} give, including the
radial constant,
\begin{equation}\label{e:near-diagonal}
 (1-s)\rho^{-p}
 \int_{B_\rho(x)}\frac{dy}{|x-y|^{n-(1-s)p}}
 =\frac{n\omega_n}{p}\rho^{-sp}
\end{equation}
(this fact is already explicit in formula (3.17) of \cite{FocardiAperiodic}).
For $|x-y|\geq\rho$, we use $|\varphi_{\ii}(x)-
\varphi_{\ii}(y)|\leq1$ and item~(ii) of the same lemma to obtain
\begin{equation}\label{e:far-field}
 (1-s)\int_{\Rn\setminus B_\rho(x)}
 \frac{dy}{|x-y|^{n+sp}}
 =\frac{n\omega_n(1-s)}{sp}\rho^{-sp}
 \leq\frac{n\omega_n}{s_0p}\rho^{-sp}.
\end{equation}
Therefore the normalized versions of (3.17) and (3.18), and hence
of (3.19), are bounded by a constant of the form
$c(n,p,s_0)m^{2p}$.

Finally, the passage from $E=U\times U$ to an arbitrary measurable
$E\subseteq U\times U$ uses only that $w_j^h=u_j$ outside $A_j^h$
and the last algebraic decomposition in the proof of
\cite[Lemma~3.8]{FocardiAperiodic}; it introduces no further
$s$-dependent constant.  Combining
\eqref{e:single-annulus-uniform},
\eqref{e:distinct-annuli-uniform},
\eqref{e:near-diagonal}, and \eqref{e:far-field} proves that the
constant in \eqref{e:technical-uniform} may indeed be chosen as
$c=c(n,p,s_0)$, uniformly in $j$ and $h$.

The uniform scaled Poincar\'e--Wirtinger inequality gives
\begin{equation}\label{e:uniform-PW}
 \sum_{\ii\in\Ieps}
 \int_{A_j^{\ii,h}}|u_j-z_j^{\ii,h}|^p\,dx
 \leq c_m(m^{-3h}\rj)^{s_jp}
 \defectj(u_j,A_j^h\times A_j^h).
\end{equation}
Therefore, \eqref{e:technical-uniform} and
\eqref{e:uniform-PW} yield
\begin{equation}\label{e:joining-local}
 \left|\defectj(u_j,E)-\defectj(w_j^h,E)\right|
 \leq c_m\defectj(u_j,U\times A_j^h).
\end{equation}

By summing up and averaging on $h$, being the $A_j^h$'s disjoint, we
find $h_j\in\{1,\ldots,N\}$ such that
\begin{equation}\label{e:averaging}
 \defectj(u_j,U\times A_j^{h_j})
 \leq\frac1N\defectj
 \left(u_j,U\times\bigcup_{h=1}^NA_j^h\right)
 \leq\frac1N\defectj(u_j,U\times U).
\end{equation}
Set $w_j:=w_j^{h_j}$.  Then \eqref{e:equal1} and
\eqref{e:equal2} are satisfied by construction, and
\eqref{e:joining-local}--\eqref{e:averaging} imply
\eqref{e:uniform-joining-error}.

To prove that $(w_j)_{j\in\N}$ converges to $u$ in $L^p(U)$, we use
again \eqref{e:uniform-PW}, with radius $m^{-3h_j}\rj$, and the very
definition of $w_j$ as a convex combination of $u_j$ and the mean
value $(u_j)_{\Cij}$ on
$\Bijhj\setminus\overline B_j^{\ii,h_j+1}$.  We obtain
\begin{align}
 \|u_j-w_j\|_{L^p(U)}^p
 &=\sum_{\ii\in\Ieps}
 \|u_j-w_j\|_{L^p(\Bijhj\setminus B_j^{\ii,h_j+1})}^p\notag\\
 &\leq\sum_{\ii\in\Ieps}
 \|u_j-(u_j)_{\Cij}\|_{L^p
 (\Bijhj\setminus B_j^{\ii,h_j+1})}^p\notag\\
 &\leq c_m\rj^{s_jp}\defectj(u_j,U\times U).
 \label{e:wj-Lp}
\end{align}
Since $\rj\to0$, $s_j\geq s_0$, and \eqref{e:energy-bound} holds,
the right-hand side is infinitesimal.  Thus
\eqref{e:wj-convergence} follows from $u_j\to u$ in $L^p(U)$.

We now fill the balls $D_j^{\ii}$ and check all the stated properties of $v_j$.  
Since $w_j=(u_j)_{\Cij}$ on
$\Cij=P_j^{\ii}\setminus\overline{D_j^{\ii}}$, the function $v_j$
is constant on the whole ball $P_j^{\ii}$.  Furthermore, by the
construction of $w_j$ and the uniform Poincar\'e--Wirtinger
inequality,
\begin{align}
 \|v_j-w_j\|_{L^p(U)}^p
 &=\sum_{\ii\in\Ieps}
 \int_{D_j^{\ii}}|w_j-(u_j)_{\Cij}|^p\,dx\notag\\
 &\leq c_m\rj^{s_jp}\defectj(u_j,U\times U).
 \label{e:vj-Lp}
\end{align}
Thus \eqref{e:vj-convergence} follows from
\eqref{e:wj-convergence}.  

We now prove the energy estimates directly
by applying Poincar\'e--Wirtinger on a fixed annulus.  This also
establishes that $v_j\in W^{s_j,p}(U)$, without a separate
geometric comparison of the interactions of the filled balls.

Set 
\begin{equation}\label{e:djGj}
 G_j:=(U\times U)\setminus
       \bigcup_{\ii\in\Ieps}(P_j^{\ii}\times P_j^{\ii}).
\end{equation}
Since $m\geq2$, the ball $B_{2\rho_j}(\xiij)$ is contained in
$P_j^{\ii}$, and $w_j=(u_j)_{\Cij}$ on
$B_{2\rho_j}(\xiij)\setminus\overline B_{\rho_j}(\xiij)$.
Applying \eqref{e:uniform-PW-scaled} to the pair
$(B_2,B_2\setminus\overline B_1)$ and scaling by $\rho_j$ gives
\begin{align}
 \rho_j^{-s_jp}\int_{D_j^{\ii}}
       |w_j-(u_j)_{\Cij}|^p\,dx
 &\leq c(1-s_j)[w_j]^p_{s_j,p;B_{2\rho_j}(\xiij)}\notag\\
 &\leq c\defectj(w_j,P_j^{\ii}\times P_j^{\ii}).
 \label{e:Lp-internal}
\end{align}
Here $c=c(n,p,s_0)$ is independent of $m,j,N$ and $\rho_j$.

For $x\in D_j^{\ii}$, the distance to $U\setminus P_j^{\ii}$
is at least $(m-1)\rho_j\geq m\rho_j/2$.  Thus
\begin{equation}\label{e:tail}
 \int_{U\setminus P_j^{\ii}}
       \frac{dy}{|x-y|^{n+s_jp}}
 \leq c(m\rho_j)^{-s_jp}.
\end{equation}
Note that $w_j=v_j$ outside the inner balls.  Using
$|a-b|^p\leq2^{p-1}(|a|^p+|b|^p)$ for interactions between
different inner balls, \eqref{e:Lp-internal} and \eqref{e:tail}
therefore imply
\begin{align}
 \defectj(w_j-v_j,G_j)
 &\leq c(1-s_j)(m\rho_j)^{-s_jp}
      \sum_{\ii\in\Ieps}\int_{D_j^{\ii}}
            |w_j-(u_j)_{\Cij}|^p\,dx\notag\\
 &\leq c(1-s_j)m^{-s_jp}\mathcal S_j.
 \label{e:dj-energy}
\end{align}
On each $P_j^{\ii}\times P_j^{\ii}$, the increments of $w_j-v_j$
equal those of $w_j$, because $v_j$ is constant on $P_j^{\ii}$.
Consequently
\[
 \defectj\Big(w_j-v_j,\bigcup_{\ii\in\Ieps}
               (P_j^{\ii}\times P_j^{\ii})\Big)=\mathcal S_j.
\]
Together with \eqref{e:dj-energy} and $w_j-v_j\in L^p(U)$, this
proves $w_j-v_j\in W^{s_j,p}(U)$ and hence
$v_j=w_j-(w_j-v_j)\in W^{s_j,p}(U)$.

It remains to compare the two energies on $G_j$.  The inequality
\[
 \bigl||a+b|^p-|a|^p\bigr|
 \leq c(p)\bigl(|a|^{p-1}|b|+|b|^p\bigr)
\]
and H\"older's inequality give, for every measurable
$E\subseteq U\times U$,
\begin{align}
 &\left|\defectj(w_j,E\cap G_j)
             -\defectj(v_j,E\cap G_j)\right|\notag\\
 &\quad\leq c\defectj(v_j,U\times U)^{(p-1)/p}
                  \defectj(w_j-v_j,G_j)^{1/p}
                  +c\defectj(w_j-v_j,G_j)\notag\\
 &\quad\leq c(1-s_j)^{1/p}m^{-s_j}
          \defectj(v_j,U\times U)^{(p-1)/p}\mathcal S_j^{1/p}
          +c(1-s_j)m^{-s_jp}\mathcal S_j\notag\\
 &\quad\leq c(1-s_j)^{1/p}m^{-s_j}
           \bigl(\defectj(v_j,U\times U)+\mathcal S_j\bigr).
 \label{e:external-difference}
\end{align}
The last step follows from Young's inequality.  On the internal
products $P_j^{\ii}\times P_j^{\ii}$ the energy of $v_j$ is zero,
and that of $w_j$ gives exactly $\mathcal S_j(E)$.  Adding these
contributions proves \eqref{e:filling-difference}; taking
$E=U\times U$ gives \eqref{e:filling-lower} and
\eqref{e:filling-upper}.

Finally, the functions $w_j$ are constructed as convex combinations
of $u_j$ and its mean values on the annuli $\Cij$.  This proves the
first inequality in \eqref{e:Linfty}; the second follows immediately
from the definition of $v_j$.
\end{proof}
\begin{remark}
The additional factor $1-s_j$ in \eqref{e:dj-energy} is the
gain specific to the present regime.  It comes from combining
the normalized tail with the uniform Poincar\'e estimate on the
annulus.  Thus $m$ could be fixed throughout
the applications.  We leave it explicit to retain the comparison
with the fixed-order arguments of \cite{FocardiAperiodic}.
\end{remark}
We introduce next the notation for Delone sets used in
\cite[Definition~2.2 and Propositions~2.4--2.5]{FocardiAperiodic}.
A set $\Lambda\subset\Rn$ is a Delone set if
\begin{equation}\label{e:Delone-definition-new}
 r_\Lambda:=\frac12\inf_{\substack{x,y\in\Lambda\\x\neq y}}|x-y|>0,
 \qquad
 R_\Lambda:=\sup_{x\in\Rn}\operatorname{dist}(x,\Lambda)<+\infty.
\end{equation}
Here $\operatorname{dist}(x,\Lambda):=\inf_{y\in\Lambda}|x-y|$.
Such a set is countably infinite; we write
$\Lambda=\{x^{\ii}\}_{\ii\in\mathbb Z^n}$.  The Vorono\"i cell
of $x^{\ii}$ is
\begin{equation}\label{e:Voronoi-definition}
 V^{\ii}:=\{x\in\Rn:|x-x^{\ii}|\leq|x-x^{\kk}|
                 \text{ for every }\kk\in\mathbb Z^n\}.
\end{equation}
The cells are convex polytopes with pairwise disjoint interiors,
cover $\Rn$, and satisfy
\begin{equation}\label{e:Voronoi-radii-new}
 B_{r_\Lambda}(x^{\ii})\subseteq V^{\ii}
 \subseteq\overline B_{R_\Lambda}(x^{\ii}).
\end{equation}
In particular, their boundaries have zero Lebesgue measure.
For a bounded open set $A\subset\Rn$ set
\begin{equation}\label{e:interior-boundary-indices-new}
 \begin{aligned}
 \mathcal I_\Lambda(A)&:=\{\ii\in\mathbb Z^n:V^{\ii}\subseteq A\},\\
 \mathscr I_\Lambda(A)&:=\{\ii\in\mathbb Z^n:
                              V^{\ii}\cap\partial A\neq\varnothing\}.
 \end{aligned}
\end{equation}
Both index sets are finite.  

Given a sequence of Delone sets $\Lambda_j=\{x_j^{\ii}\}_{\ii\in\mathbb Z^n}$
denote its Vorono\"i cells by
$V_j^{\ii}$ and write $\mathcal I_j(A):=\mathcal I_{\Lambda_j}(A)$, 
$\mathscr I_j(A):=\mathscr I_{\Lambda_j}(A)$, 
and $R_j:=R_{\Lambda_j}$.

In the following result we will use that the constant in the Poincar\'e--Wirtinger inequality is 
uniform over all closed convex sets $K$ satisfying $B_1\subset K\subset B_M$, for some $M>1$.
This is the content of Proposition~\ref{p:PW Voronoi} in the appendix.
\begin{corollary}\label{c:Delone-joining}
Let $U$, $p$, $s_j$, and $u_j$ be as in
Lemma~\ref{l:uniform-joining}.  Let $(\Lambda_j)_j$ be a sequence
of Delone sets and let $0<\eps_j\leq r_{\Lambda_j}$ satisfy
\begin{equation}\label{e:Delone-radii}
 \eps_j\longrightarrow0,
 \qquad \sup_j\frac{R_j}{\eps_j}<+\infty.
\end{equation}
Set $\Ieps:=\mathcal I_{\Lambda_j}(U)$ and fix integers
$m\geq2$, $N\geq1$.  Then Lemma~\ref{l:uniform-joining} applies
to the family $\{x_j^{\ii}\}_{\ii\in\Ieps}$.  If $h_j$ is the
sequence of indices given by that lemma, the functions
\begin{equation}\label{e:zeta}
 \zeta_j:=\sum_{\ii\in\Ieps}(u_j)_{\Cij}\chi_{\Vij}
\end{equation}
converge to $u$ strongly in $L^p(U)$.  
\end{corollary}

\begin{proof}
By \eqref{e:Voronoi-radii-new} and the definition of $\Ieps$, the
points $x_j^{\ii}$ satisfy \eqref{e:separation}.  It remains to prove
the convergence of $(\zeta_j)_j$.  Fix
$M>\sup_jR_j/\eps_j$ and set
\[
 K_j^{\ii}:=\eps_j^{-1}(V_j^{\ii}-x_j^{\ii}).
\]
These sets are convex and satisfy $B_1\subset K_j^{\ii}\subset B_M$.
Applying inequality \eqref{e:PW Voronoi} to $K_j^{\ii}$ and
$v(x)=u_j(x_j^{\ii}+\eps_jx)$, and scaling back, we obtain
\begin{equation}\label{e:cell-PW-Delone}
 \|u_j-(u_j)_{\Vij}\|_{L^p(\Vij)}^p
 \leq c \eps_j^{s_jp}\defectj(u_j,\Vij\times\Vij),
\end{equation}
where $c=c(n,p,s_0,M)$ is independent of $j$ and $\ii$.
Since $\Cij\subset\Vij$ and
$h_j\leq N$, we also have
$\Ln(\Vij)/\Ln(\Cij)\leq c(n,m,M,N)$.  Jensen's inequality thus yields
\begin{align*}
 \|u_j-(u_j)_{\Cij}\|_{L^p(\Vij)}^p
 &\leq 2^{p-1}\left(1+\frac{\Ln(\Vij)}{\Ln(\Cij)}\right)
                \|u_j-(u_j)_{\Vij}\|_{L^p(\Vij)}^p\\
 &\leq c \eps_j^{s_jp}\defectj(u_j,\Vij\times\Vij),
\end{align*}
with $c=c(n,p,s_0,m,M,N)$.  Summing over $\ii\in\Ieps$ gives
\begin{equation}\label{e:zeta-estimate}
 \sum_{\ii\in\Ieps}
 \|u_j-(u_j)_{\Cij}\|_{L^p(\Vij)}^p
 \leq c \eps_j^{s_jp}\defectj(u_j,U\times U).
\end{equation}
Set $E_j:=U\setminus\bigcup_{\ii\in\Ieps}\Vij$.  If $x\in E_j$,
any Vorono\"i cell containing $x$ also contains a point outside $U$. 
Since its diameter is at most $2R_j$,
\[
 E_j\subseteq\{x\in U:\operatorname{dist}(x,\partial U)\leq2R_j\}.
\]
As $R_j\to0$ and $U$ is bounded and open, $\Ln(E_j)\to0$.
By the definition of $\zeta_j$,
\begin{equation}\label{e:zeta-difference}
 \|\zeta_j-u_j\|_{L^p(U)}^p
 =\sum_{\ii\in\Ieps}
 \|u_j-(u_j)_{\Cij}\|_{L^p(\Vij)}^p
 +\|u_j\|_{L^p(E_j)}^p.
\end{equation}
The first term tends to zero by \eqref{e:zeta-estimate} and
\eqref{e:energy-bound}; the second does so by the equi-integrability
of $(|u_j|^p)_j$ due to the strong convergence $u_j\to u$ in $L^p(U)$. The conclusion then follows.
\end{proof}

\begin{remark}[Separated families]\label{r:random}
Lemma~\ref{l:uniform-joining} requires only
\eqref{e:separation}.  Relative density is used
in Corollary~\ref{c:Delone-joining} to prove the convergence of the
cellwise averages \eqref{e:zeta}.
\end{remark}

\begin{corollary}\label{c:combination}
Under the assumptions and with the notation of
Lemma~\ref{l:uniform-joining}, let $c_m$ be the joining constant
and let $\delta_{j,m}$ be as in \eqref{e:filling-error-factor}.
For all sufficiently large $j$,
\begin{equation}\label{e:combined-lower}
 \defectj(u_j,U\times U)
 \geq\frac{1-\delta_{j,m}}{1+c_m/N}
       \bigl(\defectj(v_j,U\times U)+\mathcal S_j\bigr).
\end{equation}
If $N>c_m$, then also
\begin{equation}\label{e:combined-upper}
 \defectj(u_j,U\times U)
 \leq\frac{1+\delta_{j,m}}{1-c_m/N}
       \bigl(\defectj(v_j,U\times U)+\mathcal S_j\bigr).
\end{equation}
\end{corollary}

\begin{proof}
Taking $E=U\times U$ in \eqref{e:uniform-joining-error} gives
\[
 \left(1-\frac{c_m}{N}\right)\defectj(u_j,U\times U)
 \leq\defectj(w_j,U\times U)
 \leq\left(1+\frac{c_m}{N}\right)\defectj(u_j,U\times U).
\]
Combining these inequalities with
\eqref{e:filling-lower}--\eqref{e:filling-upper} proves the claim.
For each fixed $m$, the filling factors tend to one as $j\to\infty$;
only the joining error requires the subsequent limit $N\to\infty$.
\end{proof}

\section{Aperiodic obstacle homogenization in the BBM regime}
\label{s:aperiodic-BBM}

Let $U\subset\Rn$ be bounded, connected, and Lipschitz, and let
$\Lambda_j=\{x_j^{\ii}\}_{\ii\in\mathbb Z^n}$ be a sequence of
Delone sets.  We retain the notation
$r_{\Lambda_j}$, $R_j=R_{\Lambda_j}$, $V_j^{\ii}$,
$\mathcal I_j(A)$, and $\mathscr I_j(A)$ introduced above.
Choose $0<\eps_j\leq r_{\Lambda_j}$ and assume
\begin{equation}\label{e:Delone-scaling-new}
 \eps_j\longrightarrow0,
 \qquad
 \sup_j\frac{R_j}{\eps_j}<+\infty,
 \qquad
 \#\mathcal I_j(U)\,\eps_j^n\longrightarrow\theta\in(0,+\infty).
\end{equation}
Finally, suppose that the empirical measures of the interior centres
satisfy
\begin{equation}\label{e:empirical-measures-new}
 \mu_j:=\frac1{\#\mathcal I_j(U)}
 \sum_{\ii\in\mathcal I_j(U)}\delta_{x_j^{\ii}}
 \stackrel{*}{\rightharpoonup}
 \beta\,\Ln\!\llcorner U
 \quad\text{in }(C_b(U))^*,
\end{equation}
where $\beta\in L^\infty(U)$, $\beta>0$ a.e., and
$\int_U\beta\,dx=1$.  Conditions
\eqref{e:Delone-scaling-new}--\eqref{e:empirical-measures-new} are
the assumptions in \cite[Section~3.1, formulas~(3.1)--(3.3)]
{FocardiAperiodic}; their
compatibility with the Delone geometry is discussed precisely in
\cite[Remark~3.2]{FocardiAperiodic} where it is shown that 
\eqref{e:empirical-measures-new} is always satisfied up to subsequences
under \eqref{e:Delone-scaling-new}.

Let $s_j\uparrow1$, with $s_j\in[s_0,1)$, and retain $p\in(1,n)$.
Fix a nonempty bounded open set $T\Subset B_1$.  Define the critical radius
and the obstacle set by
\begin{equation}\label{e:critical-obstacle-scale-new}
 \lambda_j:=\eps_j^{\frac{n}{n-s_jp}},
 \qquad
 T_j^{\ii}:=x_j^{\ii}+\lambda_jT,
 \qquad
 \mathcal K_j:=\bigcup_{\ii\in\mathbb Z^n}T_j^{\ii}.
\end{equation}
Notice that
\begin{equation}\label{e:scale-separation-new}
 \lambda_j^{n-s_jp}=\eps_j^n,
 \qquad
 \frac{\lambda_j}{\eps_j}
 =\eps_j^{\frac{s_jp}{n-s_jp}}\longrightarrow0.
\end{equation}
Thus, after discarding finitely many indices, every obstacle is
contained in the corresponding Vorono\"i cell.  The choice
\eqref{e:critical-obstacle-scale-new} is the simultaneous BBM version
of \cite[formula~(3.5)]{FocardiAperiodic}.

Define $\mathcal F_j:L^p(U)\to[0,+\infty]$ by
\begin{equation}\label{e:aperiodic-functionals-new}
 \mathcal F_j(u):=
 \begin{cases}
  (1-s_j)[u]_{s_j,p;U}^p,
  &u\in W^{s_j,p}(U),\quad
   u=0\quad\text{a.e. on }\mathcal K_j\cap U,\\
  +\infty,&\text{otherwise}.
 \end{cases}
\end{equation}
The assumption that $T$ is open allows us to use
Corollary~\ref{c:capacity-convergence}.  The extension to thin
obstacles is discussed in Remarks~\ref{r:capacity-thin-obstacles} and \ref{r:capacity-constraint}.

\begin{theorem}[Aperiodic BBM homogenization]
\label{t:aperiodic-BBM}
Under assumptions \eqref{e:Delone-scaling-new}--
\eqref{e:empirical-measures-new}, the functionals
$(\mathcal F_j)_j$ defined in \eqref{e:aperiodic-functionals-new}
$\Gamma$-converge in $L^p(U)$ to
\begin{equation}\label{e:aperiodic-limit-functional-new}
 \mathcal F(u):=
 \begin{cases}
 K_{n,p}\displaystyle\int_U|\nabla u|^p\,dx
 +\theta K_{n,p}\capp(T)
  \displaystyle\int_U|u|^p\beta\,dx,
 &u\in W^{1,p}(U),\\
 +\infty,&\text{otherwise}.
 \end{cases}
 \end{equation}
\end{theorem}
The proof follows the lines of \cite[Propositions~3.10--3.11]{FocardiAperiodic}, though in the current 
setting both the bulk energy and the capacitary term now undergo the BBM limit so that the derivation of the limiting energy contributions changes.
For the lower bound we replace formulas (3.28)--(3.29)
therein by the lower filling estimate.  For the upper bound we
insert capacitary correctors in the balls where the filled sequence
is constant and apply the upper filling estimate.
More precisely, for sequences with bounded normalized energy, 
in the lower bound, the filled sequence still converges to $u$, and its
full normalized fractional energy yields
$K_{n,p}\int_U|\nabla u|^p\,dx$.  The internal energies separated by
filling yield the additional capacitary term.  In the upper bound,
the capacitary correctors are inserted into regions where the
target function has been made constant, and the upper filling estimate
controls the interaction error.  After taking $j\to\infty$, letting the joining
parameter $N$ tend to infinity with $m$ fixed leaves the full BBM
coefficient in front of the gradient term, independently of the
relative rates of $\eps_j$ and $s_j$.

We stress that the factor $1-s_j$ simplifies the comparison with the
fixed-order arguments in \cite{FocardiAperiodic}.
Combined with the uniform Poincar\'e--Wirtinger inequality, it
makes the tails of the capacitary competitors negligible compared
with their internal energies, even when the ratio of the two
concentric radii remains fixed.  The same estimate improves the
filling error to a quantity of order $(1-s_j)^{1/p}m^{-s_j}$.
Thus the separation of the bulk and capacitary contributions is
asymptotically exact for each fixed $m\geq2$, and a further limit
in $m$ is unnecessary.  We nevertheless keep $m$ explicit to
facilitate comparison with the fixed-order constructions.

Propositions~\ref{p:aperiodic-liminf} and
\ref{p:aperiodic-limsup} below prove Theorem~\ref{t:aperiodic-BBM}.
\begin{proposition}[Compactness and $\Gamma$-liminf inequality]
\label{p:aperiodic-liminf}
Every sequence $(u_j)_j\subset L^p(U)$ such that
\begin{equation}\label{e:aperiodic-equicoercivity-assumption-new}
 \sup_j\mathcal F_j(u_j)<+\infty
\end{equation}
is relatively compact in $L^p(U)$, and all its cluster points belong
to $W^{1,p}(U)$.  Moreover, for every sequence $u_j\to u$ strongly
in $L^p(U)$,
\begin{equation}\label{e:aperiodic-liminf-new}
 \liminf_{j\to+\infty}\mathcal F_j(u_j)\geq\mathcal F(u).
\end{equation}
\end{proposition}

\begin{proof}
We first prove the lower-bound inequality for a sequence $u_j\to u$ strongly in $L^p(U)$. 
It is enough to consider a subsequence, not relabelled, on which the
left-hand side of \eqref{e:aperiodic-liminf-new} is finite and is
attained as a limit.  In particular,
\begin{equation}\label{e:aperiodic-energy-bound-new}
 \sup_j\defectj(u_j,U\times U)<+\infty.
\end{equation}
Fix $m\geq2$ and $N\in\N$.  Apply
Lemma~\ref{l:uniform-joining} to the points
$\{x_j^{\ii}\}_{\ii\in\mathcal I_j(U)}$.  Denote by $h_j$, $w_j$,
$v_j$, $P_j^{\ii}$, $D_j^{\ii}$, and $\mathcal{S}_j$ the quantities provided by
that lemma, and put
\begin{equation}\label{e:joining-values-new}
 z_j^{\ii}:=(u_j)_{C_j^{\ii,h_j}},
 \qquad
 \zeta_j:=\sum_{\ii\in\mathcal I_j(U)}
 z_j^{\ii}\chi_{V_j^{\ii}}.
\end{equation}
For fixed $m,N$, scale separation \eqref{e:scale-separation-new}
implies, for all sufficiently large $j$,
\begin{equation}\label{e:obstacle-inside-filling-new}
 T_j^{\ii}\subset
 B_{m^{-3h_j-3}\eps_j}(x_j^{\ii})\subset D_j^{\ii}
 \quad\text{for every }\ii\in\mathcal I_j(U).
\end{equation}
The joining modification $w_j$ in
Lemma~\ref{l:uniform-joining} does not change $u_j$ near the obstacles;
hence $w_j$ remains admissible for $\mathcal F_j$.  Obstacles
associated with boundary cells are untouched.  From
\eqref{e:uniform-joining-error} we obtain the analogue of
\cite[Proposition~3.10, formula~(3.27)]{FocardiAperiodic}:
\begin{equation}\label{e:aperiodic-reduction-new}
 \left(1+\frac{c_m}{N}\right)
 \liminf_{j\to+\infty}\mathcal F_j(u_j)
 \geq\liminf_{j\to+\infty}
 \defectj(w_j,U\times U).
\end{equation}

The essential change with respect to
\cite[Proposition~3.10]{FocardiAperiodic} is the use of filling to
retain the full bulk energy.  For fixed $m$, the factor
$\delta_{j,m}$ tends to zero.  By
\eqref{e:filling-lower} and the boundedness of the energy of $w_j$,
the quantities $\defectj(v_j,U\times U)+\mathcal S_j$ remain
bounded.  Since $v_j\to u$ in $L^p(U)$, the regional BBM lower
bound \cite[Theorem~8]{PonceGamma} therefore gives
\begin{align}
 \liminf_{j\to+\infty}\defectj(w_j,U\times U)
 &\geq\liminf_{j\to+\infty}\defectj(v_j,U\times U)
       +\liminf_{j\to+\infty}\mathcal S_j\notag\\
 &\geq K_{n,p}\int_U|\nabla u|^p\,dx
       +\liminf_{j\to+\infty}\mathcal S_j.
 \label{e:aperiodic-328-new}
\end{align}
This replaces \cite[Proposition~3.10, formula~(3.28)]{FocardiAperiodic}
and yields the complete BBM coefficient already at fixed $m$.
The same lower bound shows that $u\in W^{1,p}(U)$.

We now estimate the internal energies $\mathcal{S}_j$.  Recall that
\begin{equation*}
 \rho_j=m^{-3h_j-2}\eps_j,
 \qquad
 D_j^{\ii}=B_{\rho_j}(x_j^{\ii}),
 \qquad
 P_j^{\ii}=B_{m\rho_j}(x_j^{\ii}),
\end{equation*}
and that $w_j=z_j^{\ii}$ on
$P_j^{\ii}\setminus\overline{D_j^{\ii}}$.  Since $w_j=0$
a.e. on $T_j^{\ii}$, translation and
scaling in \eqref{e:interior-relative-capacity-new} give
\begin{equation}\label{e:single-cell-capacity-new}
 \defectj(w_j,P_j^{\ii}\times P_j^{\ii})
 \geq(1-s_j)\lambda_j^{n-s_jp}|z_j^{\ii}|^p
 \operatorname{cap}_{s_j,p}
 \left(T,B_{m\rho_j/\lambda_j};
             \frac{\rho_j}{\lambda_j}\right).
\end{equation}
For $z_j^{\ii}\neq0$, this follows by testing the relative capacity
with the function $(z_j^{\ii}-w_j)/z_j^{\ii}$ on $P_j^{\ii}$, extended by zero outside $P_j^{\ii}$ and 
rescaled by $\lambda_j$. The zero extension has finite global seminorm because the function 
vanishes on $P_j^{\ii}\setminus\overline D_j^{\ii}$. For $z_j^{\ii}=0$ the estimate is immediate.  
Because $h_j\in\{1,\ldots,N\}$ and \eqref{e:scale-separation-new} holds,
\begin{equation}\label{e:relative-radius-divergence-new}
 \frac{\rho_j}{\lambda_j}
 \geq m^{-3N-2}\frac{\eps_j}{\lambda_j}
 \longrightarrow+\infty.
\end{equation}

Set
\begin{equation}\label{e:Psi-new}
 \Psi_j:=\sum_{\ii\in\mathcal I_j(U)}
 \frac{\eps_j^n}{\Ln(V_j^{\ii})}\chi_{V_j^{\ii}}.
\end{equation}
The inclusions \eqref{e:Voronoi-radii-new} imply that $(\Psi_j)_j$
is bounded in $L^\infty(U)$.  Moreover,
\begin{equation}\label{e:Psi-convergence-new}
 \Psi_j\stackrel{*}{\rightharpoonup}\theta\beta
 \qquad\text{in }L^\infty(U).
\end{equation}
Indeed, for a compactly contained Lipschitz set $A\subset U$, the
difference between $\int_A\Psi_j\,dx$ and
$\eps_j^n\#\mathcal I_j(A)$ is bounded by a constant times
$\eps_j^n\#\mathscr I_j(A)$.  The latter tends to zero by
\cite[Proposition~2.5, formulas~(2.4)--(2.5)]
{FocardiAperiodic}, while
\eqref{e:Delone-scaling-new}--\eqref{e:empirical-measures-new}
identify the limit.  This is exactly the counting argument following
\cite[Proposition~3.10, formula~(3.31)]{FocardiAperiodic}.

By Corollary~\ref{c:Delone-joining},
$\zeta_j\to u$ strongly in $L^p(U)$.  Therefore
\begin{equation}\label{e:weighted-zeta-convergence-new}
 \int_U|\zeta_j|^p\Psi_j\,dx
 \longrightarrow\theta\int_U|u|^p\beta\,dx.
\end{equation}
Using \eqref{e:scale-separation-new}, summing
\eqref{e:single-cell-capacity-new}, and applying
Lemma~\ref{l:interior-capacity-BBM} with
$L_j=\rho_j/\lambda_j$, we conclude that
\begin{align}
 \liminf_{j\to+\infty}\mathcal{S}_j
 &\geq \theta K_{n,p}\capp(T)
 \int_U|u|^p\beta\,dx.
 \label{e:aperiodic-329-new}
\end{align}
Estimate \eqref{e:aperiodic-329-new} replaces
\cite[Proposition~3.10, formula~(3.29)]{FocardiAperiodic}:
the capacitary contribution is obtained from the internal energies
separated by filling.

Combining \eqref{e:aperiodic-reduction-new},
\eqref{e:aperiodic-328-new}, and \eqref{e:aperiodic-329-new}, we obtain
\begin{multline}\label{e:aperiodic-final-liminf-new}
 \left(1+\frac{c_m}{N}\right)
       \liminf_{j\to+\infty}\mathcal F_j(u_j)
 \geq K_{n,p}\int_U|\nabla u|^p\,dx
       +\theta K_{n,p}\capp(T)\int_U|u|^p\beta\,dx.
\end{multline}
Letting $N\to+\infty$ with $m$ fixed proves
\eqref{e:aperiodic-liminf-new}.  No further limit in $m$ is needed,
since both the filling error and the relative capacitary error
have already vanished as $s_j\uparrow1$.

It remains to prove the compactness assertion.  We first claim that
\eqref{e:aperiodic-equicoercivity-assumption-new} implies
\begin{equation}\label{e:aperiodic-Lp-bound-new}
 \sup_j\|u_j\|_{L^p(U)}<+\infty.
\end{equation}
Otherwise, after passing to a subsequence, set
$\widehat{u}_j:=u_j/\|u_j\|_{L^p(U)}$.  Then
\begin{equation*}
 \|\widehat{u}_j\|_{L^p(U)}=1,
 \qquad
 \mathcal F_j(\widehat{u}_j)\longrightarrow0.
\end{equation*}
The BBM compactness theorem
\cite[Theorem~1.2]{PoncePoincare} yields, up to a further
subsequence, $\widehat{u}_j\to \widehat{u}$ strongly in $L^p(U)$, with
$\|\widehat{u}\|_{L^p(U)}=1$.  The already proved liminf inequality gives
\begin{equation*}
 0\geq\mathcal F(\widehat{u})
 =K_{n,p}\int_U|\nabla\widehat{u}|^p\,dx
 +\theta K_{n,p}\capp(T)\int_U|\widehat{u}|^p\beta\,dx.
\end{equation*}
Since $T$ is nonempty and open, $\capp(T)>0$; since $\beta>0$ a.e.,
this forces $\widehat{u}=0$ a.e. on $U$, a contradiction.  Thus
\eqref{e:aperiodic-Lp-bound-new} holds.  A second application of
\cite[Theorem~1.2]{PoncePoincare}, together with
\eqref{e:aperiodic-equicoercivity-assumption-new}, gives relative
compactness in $L^p(U)$.  Finally, equations
\eqref{e:aperiodic-energy-bound-new} and
\eqref{e:aperiodic-328-new} show that every cluster point belongs to
$W^{1,p}(U)$.
\end{proof}

\begin{proposition}[The $\Gamma$-limsup inequality]
\label{p:aperiodic-limsup}
For every $u\in L^p(U)$ there exists a sequence
$(u_j)_j\subset L^p(U)$ such that $u_j\to u$ strongly in $L^p(U)$
and
\begin{equation}\label{e:aperiodic-limsup-new}
 \limsup_{j\to+\infty}\mathcal F_j(u_j)\leq\mathcal F(u).
\end{equation}
\end{proposition}

\begin{proof}
It is enough to prove the assertion for
$u\in W^{1,\infty}(U)$; the general case follows by approximation in
$W^{1,p}(U)$ and a diagonal argument, exactly as at the beginning of
the proof of \cite[Proposition~3.11]{FocardiAperiodic}.  Fix
$m\geq2$ and $N\in\N$.  We apply
Lemma~\ref{l:uniform-joining} to the constant sequence $u_j=u$ and
to the interior centres
$\{x_j^{\ii}\}_{\ii\in\mathcal I_j(U)}$.  We retain the notation
$h_j,w_j,v_j,P_j^{\ii},D_j^{\ii}$ and set
\begin{equation}\label{e:limsup-values-new}
 z_j^{\ii}:=(u)_{C_j^{\ii,h_j}},
 \qquad
 \zeta_j:=\sum_{\ii\in\mathcal I_j(U)}
 z_j^{\ii}\chi_{V_j^{\ii}}.
\end{equation}
By Lemma~\ref{l:uniform-joining} and
Corollary~\ref{c:Delone-joining}, $v_j\to u$ and $\zeta_j\to u$
strongly in $L^p(U)$.  Moreover,
\eqref{e:uniform-joining-error} and \eqref{e:filling-lower} give
\begin{equation}\label{e:limsup-bulk-new}
 \limsup_{j\to+\infty}\defectj(v_j,U\times U)
 \leq\left(1+\frac{c_m}{N}\right)
 K_{n,p}\int_U|\nabla u|^p\,dx,
\end{equation}
where we used $\delta_{j,m}\to0$ for fixed $m$ and the regional
BBM formula applied to the constant sequence $u_j=u$.

Set
\begin{equation}\label{e:limsup-radii-new}
 \rho_j:=m^{-3h_j-2}\eps_j,
 \qquad L_j:=\frac{\rho_j}{\lambda_j}.
\end{equation}
By \eqref{e:scale-separation-new}, $L_j\to+\infty$.  Choose a
$1/j$-minimizer $\varphi_j$ for
$C_{s_j,p}(T,B_{L_j})$ and truncate it so that
\begin{equation}\label{e:capacitary-corrector-new}
 0\leq\varphi_j\leq1,\qquad
 \varphi_j=0\ \text{a.e. on }B_{L_j}^c,\qquad
 \varphi_j=1\quad\text{a.e. on }T,
\end{equation}
and
\begin{equation}\label{e:corrector-energy-new}
 a_j:=(1-s_j)[\varphi_j]_{s_j,p}^p
 \leq(1-s_j)C_{s_j,p}(T,B_{L_j})+\frac1j.
\end{equation}
The convergence of rescaled relative fractional capacities to the global local counterpart
in \eqref{e:expanding-capacity-limit} of Corollary~\ref{c:capacity-convergence}  yields
\begin{equation}\label{e:corrector-capacity-limit-new}
 a_j\longrightarrow K_{n,p}\capp(T).
\end{equation}

We insert the capacitary correctors into the joining modification
$w_j$ furnished by Lemma~\ref{l:uniform-joining}.  The resulting
function is denoted by $\widehat u_j$: it agrees with $w_j$ outside
the balls $D_j^{\ii}$, while inside each such ball the profile is
replaced by a capacitary correction of the constant $z_j^{\ii}$.
Its filling is still $v_j$.  More precisely, define
\begin{equation}\label{e:interior-recovery-new}
 \widehat u_j(x):=
 \begin{cases}
 z_j^{\ii}\left(1-
 \varphi_j\left(\dfrac{x-x_j^{\ii}}{\lambda_j}\right)\right),
 &x\in D_j^{\ii},\quad \ii\in\mathcal I_j(U),\\[6pt]
 v_j(x),&x\in U\setminus
 \displaystyle\bigcup_{\ii\in\mathcal I_j(U)}D_j^{\ii}.
 \end{cases}
\end{equation}
Since $v_j=z_j^{\ii}$ on $D_j^{\ii}$ and $\varphi_j$ vanishes
outside $B_{L_j}$, this definition gives a function in
$W^{s_j,p}(U)$ which is constant on
$P_j^{\ii}\setminus\overline D_j^{\ii}$ and vanishes
a.e. on every interior obstacle
$T_j^{\ii}$.  Its filling is precisely $v_j$.  Therefore the proof
of \eqref{e:filling-upper}, applied to the pair
$(\widehat u_j,v_j)$, gives
\begin{equation}\label{e:recovery-filling-upper-new}
 \defectj(\widehat u_j,U\times U)
 \leq(1+\delta_{j,m})
 \left(\defectj(v_j,U\times U)+\widehat{\mathcal S}_j\right),
\end{equation}
where
\begin{equation*}
 \widehat{\mathcal S}_j:=\sum_{\ii\in\mathcal I_j(U)}
 \defectj(\widehat u_j,
 P_j^{\ii}\times P_j^{\ii}).
\end{equation*}
A change of variables, \eqref{e:scale-separation-new}, and
\eqref{e:Psi-new} imply
\begin{align}
 \widehat{\mathcal S}_j
 &\leq (1-s_j)\lambda_j^{n-s_jp}
 [\varphi_j]_{s_j,p}^p
 \sum_{\ii\in\mathcal I_j(U)}|z_j^{\ii}|^p\notag\\
 &=a_j\int_U|\zeta_j|^p\Psi_j\,dx.
 \label{e:limsup-capacitary-sum-new}
\end{align}
In view of \eqref{e:Psi-convergence-new},
\eqref{e:weighted-zeta-convergence-new}, and
\eqref{e:corrector-capacity-limit-new},
\begin{equation}\label{e:limsup-capacitary-limit-new}
 \limsup_{j\to+\infty}\widehat{\mathcal S}_j
 \leq\theta K_{n,p}\capp(T)
 \int_U|u|^p\beta\,dx.
\end{equation}
This proves the estimate corresponding to
\cite[Proposition~3.11, Step~1 and formula~(3.35)]{FocardiAperiodic}.

We also have $\widehat u_j\to u$ strongly in $L^p(U)$.
Indeed, the scaled uniform Poincar\'e inequality on
$(B_{2L_j},B_{2L_j}\setminus B_{L_j})$ gives
\begin{equation*}
 \|\varphi_j\|_{L^p(\Rn)}^p
 \leq c(1-s_j)L_j^{s_jp}[\varphi_j]_{s_j,p}^p.
\end{equation*}
Consequently,
\begin{align}
 \|\widehat u_j-v_j\|_{L^p(U)}^p
 &\leq c\rho_j^{s_jp}a_j\eps_j^n
 \sum_{\ii\in\mathcal I_j(U)}|z_j^{\ii}|^p\notag\\
 &=c\rho_j^{s_jp}a_j
 \int_U|\zeta_j|^p\Psi_j\,dx\longrightarrow0.
 \label{e:corrector-Lp-small-new}
\end{align}

It remains to treat the obstacles whose Vorono\"i cells meet
$\partial U$.  Let $\eta\in C_c^\infty(B_2)$ satisfy
$0\leq\eta\leq1$ and $\eta=1$ on $B_1$.  Since
$\lambda_j/\eps_j\to0$, the balls
$B_{2\lambda_j}(x_j^{\ii})$, $\ii\in\mathscr I_j(U)$, are mutually
disjoint for large $j$ and do not meet the interior joining and
filling regions.  In those balls $\widehat u_j=u$.  Define
\begin{equation}\label{e:boundary-recovery-new}
 u_j(x):=
 \begin{cases}
 \left(1-\eta\left(\dfrac{x-x_j^{\ii}}{\lambda_j}\right)\right)
 \widehat u_j(x),
 &x\in B_{2\lambda_j}(x_j^{\ii})\cap U,
 \quad\ii\in\mathscr I_j(U),\\[6pt]
 \widehat u_j(x),&\text{otherwise}.
 \end{cases}
\end{equation}
Thus $u_j$ is admissible for $\mathcal F_j$.  Writing
$(\partial U)_r:=\{x\in\Rn:\operatorname{dist}(x,\partial U)<r\}$,
the boundary counting estimate \cite[Proposition~2.5, formula~(2.5)]
{FocardiAperiodic} gives
\begin{equation}\label{e:boundary-counting-new}
 \eps_j^n\#\mathscr I_j(U)
 \leq c\,\\Ln((\partial U)_{2R_j})\longrightarrow0.
\end{equation}
The boundary correction corresponds to Steps~4--6 of
\cite[Proposition~3.11]{FocardiAperiodic}.  To estimate its cost,
set
\[
 \theta_j(x):=\max_{\ii\in\mathscr I_j(U)}
             \eta\left(\frac{x-x_j^{\ii}}{\lambda_j}\right),
\]
with $\theta_j=0$ if the index set is empty.  Then
$u_j=(1-\theta_j)\widehat u_j$ on $U$.  Since
$|\max_i a_i-\max_i b_i|^p\leq\sum_i|a_i-b_i|^p$, scaling gives
\begin{equation}\label{e:norma thetaj}
 (1-s_j)[\theta_j]_{s_j,p}^p
 \leq \eps_j^n\#\mathscr I_j(U)\,(1-s_j)[\eta]_{s_j,p}^p
 \longrightarrow0
\end{equation}
by \eqref{e:boundary-counting-new} and the BBM formula for $\eta$.
The identity
\begin{align*}
 u_j(x)-u_j(y)
 =(1-\theta_j(x))(\widehat u_j(x)-\widehat u_j(y))
 +\widehat u_j(y)(\theta_j(y)-\theta_j(x))
\end{align*}
and Minkowski's inequality imply
\[
 \defectj(u_j,U\times U)^{1/p}
 \leq\defectj(\widehat u_j,U\times U)^{1/p}
       +\|u\|_{L^\infty(U)}
             \bigl((1-s_j)[\theta_j]_{s_j,p}^p\bigr)^{1/p}.
\]
Therefore, by \eqref{e:norma thetaj} we get
\begin{equation}\label{e:boundary-energy-negligible-new}
 \defectj(u_j,U\times U)^{1/p}
 \leq\defectj(\widehat u_j,U\times U)^{1/p}+\gamma_j,
 \qquad \gamma_j\longrightarrow0.
\end{equation}
Moreover, $\Ln(\operatorname{supp}\theta_j)
\leq c\lambda_j^n\#\mathscr I_j(U)\to0$, and hence
\[
 \|u_j-\widehat u_j\|_{L^p(U)}^p
 \leq\|u\|_{L^\infty(U)}^p\Ln(\operatorname{supp}\theta_j)
 \longrightarrow0.
\]
Combining \eqref{e:limsup-bulk-new},
\eqref{e:recovery-filling-upper-new},
\eqref{e:limsup-capacitary-limit-new}, and
\eqref{e:boundary-energy-negligible-new}, and using
$\delta_{j,m}\to0$, we obtain
\begin{align}
 \limsup_{j\to+\infty}\mathcal F_j(u_j)
 \leq{}&\left(1+\frac{c_m}{N}\right)
       K_{n,p}\int_U|\nabla u|^p\,dx\notag\\
 &+\theta K_{n,p}\capp(T)\int_U|u|^p\beta\,dx.
 \label{e:limsup-mN-new}
\end{align}
Letting $N\to+\infty$ at fixed $m$ and making a diagonal choice
proves \eqref{e:aperiodic-limsup-new} for $u\in W^{1,\infty}(U)$.
Since $U$ is Lipschitz and $\beta\in L^\infty(U)$, approximation
in $W^{1,p}(U)$ and a second diagonal argument prove the claim
for every $u\in W^{1,p}(U)$.  If $u\notin W^{1,p}(U)$, the
right-hand side is infinite and there is nothing to prove.
\end{proof}
We deduce next the convergence of minimizers for a suitable family of perturbed  functionals
related to $\mathcal F_j$'s.
\begin{corollary}
Let $g\in L^{p'}(U)$, $p'$ being the conjugate exponent of $p$, and 
\[
\mathcal G_j(u):=\mathcal F_j(u)+\int_Ugu dx\,\qquad 
\mathcal G(u):= \mathcal F(u)+\int_Ugu dx\,,
\]
where $\mathcal F_j$ is defined in \eqref{e:aperiodic-functionals-new},
and $\mathcal F$ in \eqref{e:aperiodic-limit-functional-new}, respectively.
Then $(\mathcal G_j)_j$ $\Gamma(L^p(U))$-converges to $\mathcal{G}$.

Moreover, if $u_j$ minimizes $\mathcal{G}_j$ over $W^{s_j,p}(U)$, 
the sequence $(u_j)_j$ converges to the unique minimizer of $\mathcal{G}$,
and moreover $\mathcal G_j(u_j)\to\mathcal G(u)$ as $j\to+\infty$.
\end{corollary}
\begin{proof}
The $\Gamma$-convergence of $(\mathcal G_j)_j$ to $\mathcal G$ is standard in view of
Theorem~\ref{t:aperiodic-BBM} and the linearity of the perturbation.

We claim next that there is $\lambda_0>0$ such that for every 
 $\lambda\in(0,\lambda_0)$ there is $j_\lambda$ such that for $j\geq j_\lambda$
 \begin{equation}\label{e:coercivity Fj conv min} 
 \mathcal F_j(u)\geq\lambda\|u\|_{L^p(U)}^p
 \end{equation}
 for every $u\in L^p(U)$.  The proof is similar to that of \eqref{e:aperiodic-Lp-bound-new}.
 We detail it for the reader's convenience. 
 Assume not, then there are a subsequence of integers $j_k$ and $\lambda_{k}\to 0^+$ such that
 \[
 \mathcal F_{j_k}(v_{j_k})<\lambda_k\|v_{j_k}\|_{L^p(U)}^p
 \]
 for some $v_{j_k}\in L^p(U)$. Moreover, up to a normalization, we may also assume 
 $\|v_{j_k}\|_{L^p(U)}=1$. Then, Proposition~\ref{p:aperiodic-liminf} implies that a subsequence (not relabeled) converges strongly in $L^p(U)$ to some function $v$, with $\|v\|_{L^p(U)}=1$, and
 \[
\theta K_{n,p}\capp(T)
\int_U|v|^p\beta dx\leq\mathcal{F}(v)\leq\liminf_k \mathcal F_{j_k}(v_{j_k})\leq 0\,.
\]
Being $\beta>0$ a.e. on $U$ (cf. the comments right after \eqref{e:empirical-measures-new}) and 
$\theta K_{n,p}\capp(T)>0$,  we conclude $v=0$ a.e. on $U$, a contradiction.

With \eqref{e:coercivity Fj conv min}  at hand, the equi-coercivity of the family $(\mathcal G_j)_j$ follows easily from Young's inequality and Proposition~\ref{p:aperiodic-liminf}. The rest of the statement is an easy consequence of the fundamental theorem of $\Gamma$-convergence (cf. \cite[Theorem~7.8]{DalMaso-book}), 
the strict convexity of $\mathcal{G}$ and the Urysohn's property. 
\end{proof}
We record next the corresponding result when the obstacles are
scaled according to the critical radius of the local problem.  Set
\begin{equation}\label{e:local-critical-obstacle-scale}
 \bar\lambda_j:=\eps_j^{\frac{n}{n-p}},
 \qquad
 \bar T_j^{\ii}:=x_j^{\ii}+\bar\lambda_jT,
 \qquad
 \bar{\mathcal K}_j:=
 \bigcup_{\ii\in\mathbb Z^n}\bar T_j^{\ii},
\end{equation}
and define $\bar{\mathcal F}_j:L^p(U)\to[0,+\infty]$ by
\begin{equation}\label{e:local-scale-functionals}
 \bar{\mathcal F}_j(u):=
 \begin{cases}
  (1-s_j)[u]_{s_j,p;U}^p,
  &u\in W^{s_j,p}(U),\quad
   u=0\quad\text{a.e. on }
   \bar{\mathcal K}_j\cap U,\\
  +\infty,&\text{otherwise}.
 \end{cases}
\end{equation}

\begin{theorem}[The three regimes at the local critical scale]
\label{t:aperiodic-local-critical-scale}
Assume \eqref{e:Delone-scaling-new}--
\eqref{e:empirical-measures-new}, and suppose that
\begin{equation}\label{e:rate-regimes}
 (1-s_j)|\log \eps_j|\longrightarrow\ell
 \qquad\text{for some }\ell\in[0,+\infty].
\end{equation}
Then $(\bar{\mathcal F}_j)_j$ $\Gamma$-converges in $L^p(U)$ to
\begin{equation}\label{e:local-scale-limit-functional}
 \bar{\mathcal F}^{(\ell)}(u):=
 \begin{cases}
 K_{n,p}\displaystyle\int_U|\nabla u|^p\,dx
 +\theta\kappa_{\ell}K_{n,p}\capp(T)
  \displaystyle\int_U|u|^p\beta\,dx,
 &u\in W^{1,p}(U),\\
 +\infty,&\text{otherwise},
 \end{cases}
\end{equation}
where
\begin{equation}\label{e:regime-coefficient}
 \kappa_{\ell}:=
 \begin{cases}
  \displaystyle\exp\left(-\frac{np}{n-p}\ell\right),
       &\ell<+\infty,\\
  0,&\ell=+\infty.
 \end{cases}
\end{equation}
For $\ell<+\infty$ the family
$(\bar{\mathcal F}_j)_j$ is equicoercive in $L^p(U)$ after discarding finitely many indices. 
\end{theorem}
We stress again that the capacitary term is unchanged if $\ell=0$, is modulated by
$\exp(-np\ell/(n-p))$ if $\ell\in(0,+\infty)$, and disappears if $\ell=+\infty$.  

The proof is omitted, since it follows from that of
Theorem~\ref{t:aperiodic-BBM} by replacing $\lambda_j$ with
$\bar\lambda_j$ and observing that by the very definition in 
\eqref{e:regime-coefficient}
\begin{equation}\label{e:local-scale-capacity-factor}
 \eps_j^{-n}\bar\lambda_j^{n-s_jp}
 =\eps_j^{\frac{np(1-s_j)}{n-p}}
 \longrightarrow\kappa_{\ell}.
\end{equation}

Few remarsk are in order.
\begin{remark}[The case of thin obstacles]\label{r:capacity-constraint}
For open obstacles all constraints in this paper are imposed almost everywhere with respect to Lebesgue measure.  
The asymptotic result in Theorem~\ref{t:aperiodic-BBM} also extends to compact sets 
$T\Subset B_1$, possibly of zero Lebesgue measure, provided that
\begin{equation}\label{e:thin-capacity-assumption}
 (1-s_j)\operatorname{cap}_{s_j,p}(T)
 \longrightarrow K_{n,p}\capp(T),
\end{equation}
in view of the observations in Remark~\ref{r:capacity-thin-obstacles} and the proof of Theorem~\ref{t:aperiodic-BBM}. Recall that for such sets the capacities are understood in the usual outer sense: 
\[
 \operatorname{cap}_{s,p}(T)
 :=\inf\{\operatorname{cap}_{s,p}(O):T\subset O,
                                  \ O\text{ bounded and open}\}.
\]
The obstacle condition must then be imposed on the quasicontinuous
representative $\widetilde u$, namely
\[
 \widetilde u=0\quad\operatorname{cap}_{s_j,p}\text{-q.e. on }
                         \mathcal K_j\cap U.
\]
Here q.e. (quasi-everywhere) means outside a set of zero capacity of the indicated
order, and $\widetilde u$ is the quasicontinuous representative of $u$ (cf. \cite[Section~2 and Appendix~A]{FocardiPalatucciZeppieri} for the relevant definitions and for a concise account of the theory of fractional capacities). 
Likewise, the admissibility condition in the capacitary
problems is $\widetilde\varphi\geq1$ q.e. on $T$,
with respect to $\operatorname{cap}_{s,p}$ or $\capp$ as appropriate.
We stress that on open sets these conditions agree with the almost everywhere
formulation used in the text.
Note that for compact thin $T$, smooth compactly supported capacitary competitors replace the almost everywhere
competitors used for open sets.  

The compactness conclusion in Proposition~\ref{p:aperiodic-liminf} requires a positive limiting capacitary coefficient. In particular, one must assume $\capp(T)>0$.  Instead, the $\Gamma$-convergence statements remain valid even 
when this capacity is zero.

To see that the homogenization proofs still apply, let $L_j\to+\infty$.  The lower bound
$C_{s_j,p}(T,B_{L_j})\geq\operatorname{cap}_{s_j,p}(T)$ and
\eqref{e:thin-capacity-assumption} give the required capacitary
liminf.  For the upper bound, choose smooth compactly supported
functions that are at least one in a neighbourhood of $T$ and
whose local energies approach $\capp(T)$.  Each such function is
admissible in $B_{L_j}$ for large $j$, and the BBM formula gives
\[
 (1-s_j)C_{s_j,p}(T,B_{L_j})\longrightarrow K_{n,p}\capp(T).
\]
The interior comparison
\eqref{e:interior-global-comparison-new} also holds with the q.e.
constraint.  These are the capacity estimates used in the proofs of Theorem~\ref{t:aperiodic-BBM}. 

As for open sets, the joining modification in
Lemma~\ref{l:uniform-joining} preserves the constraint on the
selected obstacles.  Filling instead replaces the function by
the annular constant inside each inner ball and removes that
constraint, thereby separating the bulk energy from the
capacitary cost.  In the upper bound the capacitary correctors
impose the constraint again.

The same argument applies to Theorem~\ref{t:aperiodic-local-critical-scale}; at the local critical scale, compactness also requires $\ell<+\infty$.
\end{remark}

\begin{remark}[Translation-invariant homogeneous kernels]
\label{r:homogeneous-kernels}
Theorems~\ref{t:aperiodic-BBM} and~\ref{t:aperiodic-local-critical-scale} extend 
to energies with translation-invariant kernels, homogeneous of
degree $-(n+s_jp)$ and uniformly comparable to the Gagliardo kernel.
More precisely, let
\begin{equation}\label{e:homogeneous-kernels}
 \mathscr K_j(z):=
 \frac{a_j(z/|z|)}{|z|^{n+s_jp}},\qquad z\in\Rn\setminus\{0\},
 \qquad 0<c\leq a_j\leq C<+\infty,
\end{equation}
where $a_j\in L^\infty(\mathbb S^{n-1})$ are even and the bounds
are independent of $j$.  Note that since
replacing a kernel by its even part leaves the energy unchanged, the 
symmetry requirement is no restriction.
Assume in addition that
\begin{equation}\label{e:angular-kernel-limit}
 a_j\stackrel{*}{\rightharpoonup}a
       \quad\text{in }L^\infty(\mathbb S^{n-1}).
\end{equation}
This includes, in particular, a fixed angular profile $a_j=a$.
Define the local density and its homogeneous capacity by
\begin{align}
 \Phi_a(\xi)&:=\frac1p\int_{\mathbb S^{n-1}}
        a(\sigma)|\xi\cdot\sigma|^p\,d\mathcal H^{n-1}(\sigma),
        \qquad \xi\in\Rn,                                      \label{e:anisotropic-density}\\
 \operatorname{cap}_{\Phi_a}(T)&:=\inf\left\{
   \int_{\Rn}\Phi_a(\nabla v)\,dx:
   v\in\Dhom^{1,p}(\Rn),\ v\geq1\text{ a.e. on }T\right\}.
                                                               \label{e:anisotropic-capacity}
\end{align}
If the energies in \eqref{e:aperiodic-functionals-new} are replaced by
\[
 (1-s_j)\int_{U\times U}|u(x)-u(y)|^p
                         \mathscr K_j(x-y)\,dx\,dy,
\]
with the same obstacle constraints, their $\Gamma$-limit is finite on $W^{1,p}(U)$ with value
\begin{align*}
 &\int_U\Phi_a(\nabla u)\,dx
  +\theta\operatorname{cap}_{\Phi_a}(T)
                       \int_U|u|^p\beta\,dx,
\end{align*}
and is $+\infty$ otherwise on $L^p(U)$. At the local critical scale $\bar\lambda_j$, assuming \eqref{e:rate-regimes}, the limit on $W^{1,p}(U)$ is instead
\[
 \int_U\Phi_a(\nabla u)\,dx
 +\theta\kappa_\ell\operatorname{cap}_{\Phi_a}(T)
       \int_U|u|^p\beta\,dx.
\]
For $a=1$ we recover $\Phi_a(\xi)=K_{n,p}|\xi|^p$ and
$\operatorname{cap}_{\Phi_a}(T)=K_{n,p}\capp(T)$.

Indeed, uniform comparison in \eqref{e:homogeneous-kernels}
preserves the Sobolev and Poincar\'e estimates and the joining and
filling estimates, with constants depending also on $c$ and $C$.
The BBM limit in this case is the anisotropic density $\Phi_a$ (cf.
\cite[Theorems~1.1--1.2]{GennaioliStefani}).  Applying the argument of
Proposition~\ref{p:homogeneous-Gamma} and
Corollary~\ref{c:capacity-convergence} to these energies gives the
corresponding homogeneous and relative capacity limits.
Homogeneity yields the same factor $\lambda_j^{n-s_jp}$ under
rescaling, so the critical obstacle scales and the capacitary
weights are unchanged.  These are the ingredients needed in the
two homogenization proofs.

Under the homogeneity assumption, condition~\eqref{e:angular-kernel-limit} identifies a single limit. If this condition is omitted, the uniform bounds on $a_j$ give the result along subsequences with weakly-$*$ convergent angular profiles, whose limits may differ.
\end{remark}

The same asymptotics also hold imposing homogeneous Dirichlet boundary conditions. We introduce standard notation: for $s\in(0,1)$, let $W^{s,p}_0(U)$ denote the norm closure of
$C^\infty_c(U)$ in $W^{s,p}(U)$. 
\begin{corollary}[On homogeneous Dirichlet boundary conditions]
\label{c:homogeneous-Dirichlet}
Let $\mathcal F^0_j$ be defined as $\mathcal F_j$ in
\eqref{e:aperiodic-functionals-new}, with $W^{s_j,p}(U)$ there replaced
by $W^{s_j,p}_0(U)$.  Then the analogous conclusion of
Theorem~\ref{t:aperiodic-BBM} holds with $\Gamma(L^p(U))$-limit
\[
 \mathcal F^0(u):=
 \begin{cases}
  \mathcal F(u),&u\in W^{1,p}_0(U),\\
  +\infty,&\text{otherwise in }L^p(U).
 \end{cases}
\]
\end{corollary}
\begin{proof}
First note that since $p\in(1,n)$ and $s_j\to1^-$,
we have $s_jp\in(1,n)$ for all sufficiently large $j$. In this range,
$W^{s_j,p}_0(U)$ is the subspace of functions with zero trace on
$\partial U$.

The only difficulty is to check that the boundary condition is preserved in the limit.
To this aim fix $s_0\in(1/p,1)$, then the fractional Hardy inequality on a bounded 
Lipschitz set gives for every $ s\in[s_0,1)$ and $v\in W^{s,p}_0(U)$
\begin{equation}\label{e:uniform-boundary-Hardy}
 \int_U\frac{|v(x)|^p}{\mathrm{dist}^{sp}(x,\partial U)}\,dx
 \leq c(1-s)[v]_{s,p;U}^p\,,
\end{equation}
where $c=c(n,p,s_0,U)$ is independent of $s$.
Indeed, this follows from \cite[Theorem~1.2]{LossSloane}: we claim that
the Lipschitz regularity of $U$ allows the distance weight
there to be compared with $\mathrm{dist}^{-sp}(\cdot,\partial U)$, 
while the explicit Hardy constant in \cite[formula~(15)]{LossSloane} is of order
$(1-s)^{-1}$ as $s\uparrow1$.  

To prove the claim, we examine first the constant 
\[
D_{n,p,sp}=2\pi^{\frac{n-1}{2}}\frac{\Gamma(\frac{1+sp}2)}{\Gamma(\frac{n+sp}2)}
\int_0^1\frac{|1-r^{s-\frac1p}|^p}{(1-r)^{1+sp}}dr
\]
there. In the formula above $\Gamma$ denotes the Euler Gamma function.  
By an elementary calculation it is easy to check that 
\[
\lim_{s\to 1^-}(1-s)D_{n,p,sp}=
2\pi^{\frac{n-1}{2}}\frac{\Gamma(\frac{1+p}2)}{\Gamma(\frac{n+p}2)}
\frac1p\Big(1-\frac1p\Big)^p.
\]
Instead, to compare the weight in that theorem with the distance to
the boundary, write $d_U(x):=\operatorname{dist}(x,\partial U)$
and, for $\sigma\in\mathbb S^{n-1}$, set
\[
 d_{\sigma,U}(x):=\inf\{|t|:x+t\sigma\notin U\}.
\]
For $\alpha\in[s_0p,p]$, the weight in
\cite[formula~(14)]{LossSloane} is defined by
\[
 \begin{aligned}
 A_{n,\alpha}&:=\int_{\mathbb S^{n-1}}|\sigma_n|^\alpha\,
                       d\mathcal H^{n-1}(\sigma),\\
 m_\alpha^{-\alpha}(x)&:=\frac1{A_{n,\alpha}}
       \int_{\mathbb S^{n-1}}d_{\sigma,U}(x)^{-\alpha}\,
                       d\mathcal H^{n-1}(\sigma).
 \end{aligned}
\]
Since $d_{\sigma,U}(x)\geq d_U(x)$, we immediately obtain
$m_\alpha^{-\alpha}(x)\leq \mathcal H^{n-1}(\mathbb S^{n-1})
A_{n,\alpha}^{-1} d_U^{-\alpha}(x)$.  

The reverse inequality follows from the Lipschitz regularity of $U$.
In particular, we recall that for every $z\in\partial U$ there are $\rho>0$ 
and a closed cone $K_z$, with vertex in $z$ 
and opening independent of $z$, such that $B_{\rho}(z)\cap\partial U\cap K_z =\{z\}$, 
and $K_z\subseteq B_{\rho}(z)\cap U^c$.

Now fix $x\in U$, and choose $z\in\partial U$
with $|x-z|=d_U(x)$.  There are $\delta>0$ and $\eta\in(0,1)$,
depending only on $U$, such that, if $d_U(x)<\delta$, there is $q\in\R^n$ 
with $|q-z|=2d_U(x)$ and $B_{\eta d_U(x)}(q)\subset K_z$.
In particular, $B_{\eta d_U(x)}(q)\subset U^c$ and $|q-x|\leq3d_U(x)$.
Let $\Sigma_x\subset\mathbb S^{n-1}$ be the set of directions
whose rays from $x$ meet this ball.  Since its radius divided
by $|q-x|$ is at least $\eta/3$, we have
$\mathcal H^{n-1}(\Sigma_x)\geq a>0$, with $a$ independent
of $x$.  For every $\sigma\in\Sigma_x$,
$d_{\sigma,U}(x)\leq |q-x|+\eta d_U(x)\leq4d_U(x)$, and hence
\[
 m_\alpha(x)^{-\alpha}
 \geq\frac{a\,4^{-\alpha}}{A_{n,\alpha}}d_U(x)^{-\alpha}
 \geq c\,d_U(x)^{-\alpha}.
\]
If $d_U(x)\geq d_*$, the same conclusion follows from
$d_{\sigma,U}(x)\leq\operatorname{diam}U$ for every $\sigma$.
Thus the two weights are comparable on all of $U$, uniformly
in $\alpha\in[s_0p,p]$, since $A_{n,\alpha}$ is continuous
and strictly positive on this interval.
Taking $\alpha=sp$,  \cite[Theorem~1.2]{LossSloane} gives,
for $v\in C^\infty_c(U)$,
\[
 \int_U\frac{|v|^p}{\mathrm{dist}^{sp}(x,\partial U)}\,dx
 \leq C D_{n,p,sp}^{-1}[v]_{s,p;U}^p
 \leq C(1-s)[v]_{s,p;U}^p.
\]
The latter estimate extends from $C^\infty_c(U)$ to $W^{s,p}_0(U)$
by density. 

Thus, if $\sup_j\mathcal F^0_j(u_j)<+\infty$,
Proposition~\ref{p:aperiodic-liminf} gives, up to a subsequence,
strong convergence in $L^p(U)$ to some $u\in W^{1,p}(U)$.
By \eqref{e:uniform-boundary-Hardy} and Fatou's lemma,
$u/\mathrm{dist}(\cdot,\partial U)\in L^p(U)$, 
which implies that $u$ has zero trace.
Thus $u\in W^{1,p}_0(U)$, and the liminf inequality follows
from Proposition~\ref{p:aperiodic-liminf}.

For the limsup inequality, by density and a diagonal argument it suffices to consider
$u\in C^\infty_c(U)$.  We show that the recovery sequence constructed in
Proposition~\ref{p:aperiodic-limsup} belongs to $W^{s_j,p}_0(U)$.
Indeed, the joining and filling constructions in
Lemma~\ref{l:uniform-joining} modify $u$ only in the balls
$B_j^{\ii,h_j}$, whose radii are at most $\eps_j$.  If such a ball
is disjoint from $\operatorname{supp}u$, both $u$ and the corresponding
annular average vanish there.  Therefore, for all large $j$,
the resulting functions have support in a fixed compact subset
of $U$.  The capacitary corrections in
\eqref{e:interior-recovery-new} preserve this property, since they
are multiplied by the corresponding annular averages, which
vanish near $\partial U$.  In this setting the boundary correction is actually
unnecessary.
\end{proof}
\begin{remark}
We stress that the same argument applies to the functionals $\bar{\mathcal F}_j^0$
defined as $\bar{\mathcal F}_j$ in \eqref{e:local-scale-functionals} on 
$W^{s_j,p}_0(U)$ and $+\infty$ otherwise. The $\Gamma$-limit 
coincides with the functional  $\bar{\mathcal F}$ in Theorem~\ref{t:aperiodic-local-critical-scale}
restricted to $W^{1,p}_0(U)$ and $+\infty$ elsewhere.

Note that in this Dirichlet setting, equicoercivity also holds when
$\ell=+\infty$.  Indeed, \eqref{e:uniform-boundary-Hardy} and
the boundedness of $U$ give a uniform $L^p(U)$ bound from the
energy bound, so the BBM compactness theorem applies in every
regime.
\end{remark}

\begin{remark}
The same asymptotics hold for the Gagliardo seminorm on the full space.
More precisely, let
\[
\hat{\mathcal F}_j(u):=(1-s_j) \int_{\R^n\times\R^n}\frac{|u(x)-u(y)|^p}{|x-y|^{n+s_jp}}\,dx\,dy
\]
if $u\in W^{s_j,p}_0(U)$, $u=0$ a.e. on $\mathcal{K}_j\cap U$, and 
$+\infty$ otherwise on $L^p(U)$.

Noting that 
\[
\hat{\mathcal F}_j(u)={\mathcal F}_j^0(u)+
2(1-s_j)\int_{U}\int_{U^c}\frac{|u(x)|^p}{|x-y|^{n+s_jp}}dydx
\]
Corollary~\ref{c:homogeneous-Dirichlet} easily implies that
$(\hat{\mathcal F}_j)_j$ $\Gamma(L^p(U))$-converges to $\mathcal F^0$.
Indeed, by splitting $U^c$ into $U^c\cap B_{2R}$ and the complementary set,
where $U\subseteq B_R$, it suffices to note that the uniform Hardy inequality in 
\eqref{e:uniform-boundary-Hardy} yields
\begin{align*}
&\int_{U}\int_{U^c}\frac{|u(x)|^p}{|x-y|^{n+s_jp}}dydx\\&\leq
\frac{2^{n}n\mathcal L^n(B_1)}{s_jpR^{s_jp}}\int_U|u(x)|^pdx
+(2R)^n\mathcal L^n(B_1)\int_{U}\frac{|u(x)|^p}{\mathrm{dist}^{n+s_jp}(x,\partial U)}dx\\
&\leq\frac{C}{s_jpR^{s_jp}}\int_U|u(x)|^pdx
+CR^n(1-s_j){\mathcal F}_j^0(u) 
\end{align*}
for some dimensional constant $C>0$.
\end{remark}

\appendix
\section{Elementary results}
\label{s:singular-kernel-bounds}

We record the elementary estimates used in the proof of the uniform
joining lemma.  They are the explicit-constant version of
\cite[Lemma~A.1]{FocardiAperiodic}.

\begin{lemma}\label{l:singular-kernel-bounds}
Let $\nu>0$, $z\in\Rn$, and let $O\subset\Rn$ be measurable.
Then the following statements hold:
\begin{itemize}
\item[(i)] if $\nu\in(0,n)$, then
 \begin{equation}\label{e:singular-kernel-near}
  \int_O\frac{dx}{|x-z|^\nu}
  \leq \frac{n}{n-\nu}\omega_n^{\nu/n}
  \bigl(\Ln(O)\bigr)^{1-\nu/n};
 \end{equation}
\item[(ii)] if $\nu\in(n,+\infty)$ and
$d:=\operatorname{dist}(z,O)>0$, then
 \begin{equation}\label{e:singular-kernel-far}
  \int_O\frac{dx}{|x-z|^\nu}
  \leq\frac{n\omega_n}{\nu-n}d^{n-\nu}.
 \end{equation}
\end{itemize}
\end{lemma}

\begin{proof}
We may assume $0<\Ln(O)<+\infty$ in item~(i), since the other
cases are immediate.  By Cavalieri's formula and the change of
variables $t=r^{-\nu}$,
\begin{equation*}
 \int_O|x-z|^{-\nu}\,dx
 =\nu\int_0^{+\infty}
 \frac{\Ln(O\cap B_r(z))}{r^{1+\nu}}\,dr.
\end{equation*}
Set $\bar r=(\Ln(O)/\omega_n)^{1/n}$.  Using
$\Ln(O\cap B_r(z))\leq\min\{\omega_nr^n,\Ln(O)\}$ and splitting
the last integral at $\bar r$, we obtain
\begin{align*}
 \int_O|x-z|^{-\nu}\,dx
 &\leq\nu\omega_n\int_0^{\bar r}r^{n-\nu-1}\,dr
 +\nu\Ln(O)\int_{\bar r}^{+\infty}r^{-\nu-1}\,dr\\
 &=\frac{n}{n-\nu}\omega_n^{\nu/n}
 \bigl(\Ln(O)\bigr)^{1-\nu/n}.
\end{align*}
This proves item~(i).  For item~(ii), $O\subset\Rn\setminus B_d(z)$
up to a negligible subset of the sphere $\partial B_d(z)$, and hence
\begin{equation*}
 \int_O|x-z|^{-\nu}\,dx
 \leq\int_{\Rn\setminus B_d(z)}|x-z|^{-\nu}\,dx
 =\frac{n\omega_n}{\nu-n}d^{n-\nu}. 
\end{equation*}
\end{proof}

We show next that the constant in the Poincar\'e--Wirtinger inequality is
uniform for the large class of closed convex sets $K$ satisfying $B_1\subset K\subset B_M$.
\begin{proposition}\label{p:PW Voronoi}
Let $p\in(1,+\infty)$, $s_0\in(0,1)$, and $M>1$.
There exists a constant $c=c(n,p,s_0,M)>0$ such that, for every closed
convex set $K\subset\Rn$ with $B_1\subset K\subset B_M$, every
$s\in[s_0,1)$, and every $v\in W^{s,p}(\operatorname{int}K)$,
\begin{align}\label{e:PW Voronoi}
 \|v-(v)_K\|_{L^p(K)}^p\leq c(1-s)
 \int_{K\times K}\frac{|v(x)-v(y)|^p}{|x-y|^{n+sp}}\,dx\,dy.
\end{align}
\end{proposition}
\begin{proof}
Let
\[
 g_K(x):=\inf\{t>0:x\in tK\}
\]
be the Minkowski functional of $K$.  Convexity gives subadditivity, i.e.
$g_K(x+y)\leq g_K(x)+g_K(y)$, whereas the two inclusions $B_1\subset K\subset B_M$ imply
\[
 \frac{|x|}{M}\leq g_K(x)\leq |x|.
\]
Applying subadditivity to $x=y+(x-y)$ and then interchanging
$x$ and $y$, we obtain
$|g_K(x)-g_K(y)|\leq|x-y|$.  Thus, on the unit sphere
$\mathbb S^{n-1}:=\{x\in\Rn:|x|=1\}$, the function
$\rho_K:=1/g_K$ satisfies
\[
 1\leq\rho_K\leq M,
 \qquad
 |\rho_K(\omega)-\rho_K(\eta)|\leq M^2|\omega-\eta|.
\]
Define $F_K(0)=0$ and
\[
 F_K(t\omega):=t\rho_K(\omega)\omega,
 \qquad 0<t<1,\quad \omega\in\mathbb S^{n-1}.
\]
This map is a bijection from $B_1$ onto $\operatorname{int}K$,
with inverse $F_K^{-1}(y)=g_K(y/|y|)y$ for $y\neq0$.
For nonzero $x,y$, the inequality
\[
 \min\{|x|,|y|\}
 \left|\frac{x}{|x|}-\frac{y}{|y|}\right|\leq2|x-y|
\]
gives
\[
 \operatorname{Lip}(F_K)\leq M+2M^2,
 \qquad \operatorname{Lip}(F_K^{-1})\leq3.
\]
Consequently both constants are bounded by a number $L=L(M)\geq1$,
independently of $K$.

Given $s\in[s_0,1)$ and $v\in W^{s,p}(\operatorname{int}K)$,
put $w=v\circ F_K$.  For every constant $b\in\R$, Jensen's
inequality gives
$\|v-(v)_K\|_{L^p(K)}\leq2\|v-b\|_{L^p(K)}$.
Choose $b=(w)_{B_1}$ and apply \eqref{e:uniform-PW-prelim} on
$B_1$.  Changing variables, and using the bounds on the Jacobians
of $F_K$ and $F_K^{-1}$, yields
\begin{align*}
 \|v-(v)_K\|_{L^p(K)}^p
 &\leq 2^p L^n\|w-(w)_{B_1}\|_{L^p(B_1)}^p\notag\\
 &\leq c(n,p,s_0)L^n(1-s)
 \int_{B_1\times B_1}
       \frac{|w(x)-w(y)|^p}{|x-y|^{n+sp}}\,dx\,dy\notag\\
 &\leq c(n,p,s_0)L^{4n+sp}(1-s)
 \int_{K\times K}
       \frac{|v(x)-v(y)|^p}{|x-y|^{n+sp}}\,dx\,dy.
\end{align*}
Since $sp\leq p$, the last constant depends only on $n,p,s_0,M$.
\end{proof}

\section*{Acknowledgments}
The author would like to thank Caterina Zeppieri for several discussions, comments 
and suggestions which improved the original version of the paper.

\section*{Use of artificial intelligence}
The mathematical contents and the proofs of the results contained
in this paper were conceived exclusively by the author.

OpenAI's ChatGPT was used in the preparation of this manuscript
for drafting, revising and simplifying parts of the proofs, and for 
comparing them with the source paper \cite{FocardiAperiodic}.
It also assisted with checking calculations and references,
reorganizing the exposition, and editing.  This work involved repeated mathematical
feedback and corrections from the author, who retains all the responsibility for the contents.

\end{document}